\documentclass[12pt,reqno]{amsart} 

\usepackage{times}

\usepackage[
	top=1in, bottom=1in,
 	left=1.5in,right=1.5in
]{geometry}

\usepackage{amssymb,amsfonts,amsthm}
\usepackage{comment}
\usepackage{enumitem}
\usepackage[title]{appendix}

\usepackage[dvipsnames]{color}

\usepackage[colorlinks=true, pdfstartview=FitV, linkcolor=black, citecolor=blue, urlcolor=blue]{hyperref}

\def\Xint#1{\mathchoice
{\XXint\displaystyle\textstyle{#1}}%
{\XXint\textstyle\scriptstyle{#1}}%
{\XXint\scriptstyle\scriptscriptstyle{#1}}%
{\XXint\scriptscriptstyle%
\scriptscriptstyle{#1}}%
\!\int}
\def\XXint#1#2#3{{\setbox0=\hbox{$#1{#2#3}{%
\int}$ }
\vcenter{\hbox{$#2#3$ }}\kern-.6\wd0}}
\def\barint{\, \Xint -} 
\def\bariint{\barint_{} \kern-.4em \barint}
\def\bariiint{\bariint_{} \kern-.4em \barint}
\renewcommand{\iint}{\int_{}\kern-.34em \int} 
\renewcommand{\iiint}{\iint_{}\kern-.34em \int} 

\DeclareMathAlphabet{\mathcal}{OMS}{cmsy}{m}{n}

\theoremstyle{plain}

\newtheorem{theorem}{Theorem}[section]

\newtheorem{definition}[theorem]{Definition}
\newtheorem{lemma}[theorem]{Lemma}

\newtheorem{proposition}[theorem]{Proposition}

\theoremstyle{definition}
\newtheorem{remark}[theorem]{Remark}

\newcommand{\R}{\mathbb{R}}

\newcommand{\N}{\mathbb{N}}

\newcommand{\bP}{\mathbb{P}}

\newcommand{\p}{\partial}

\newcommand{\les}{\lesssim}
\newcommand{\ges}{\gtrsim}
\newcommand{\norm}[1]{\lVert #1 \rVert}
\renewcommand{\:}{\colon}

\newcommand{\wto}{\rightharpoonup} 
\newcommand{\wstar}{\overset{\ast}{\rightharpoonup}}
\newcommand{\into}{\hookrightarrow}

\newcommand{\uloc}{\mathrm{uloc}}

\newcommand{\loc}{{\rm loc}}

\newcommand{\para}{{\rm par}}

\let\div\relax
\DeclareMathOperator{\div}{div}

\let\tilde\relas
\newcommand{\tilde}[1]{\widetilde{#1}}

\newcommand{\osc}{{\rm osc}}

\newcommand{\harm}{{\rm harm}}

\definecolor{darkgreen}{rgb}{0,0.5,0}
\definecolor{darkblue}{rgb}{0,0,0.7}
\definecolor{darkred}{rgb}{0.9,0.1,0.1}
\definecolor{lightblue}{rgb}{0,0.51,1}

\numberwithin{equation}{section}
\setlist[enumerate]{leftmargin=*}

\title[Localized smoothing and concentration in the half space]{Localized smoothing and concentration for the Navier-Stokes equations in the half space}
\author[D. Albritton]{Dallas Albritton}
\address[D. Albritton]{School of Mathematics, Institute for Advanced Study, 1 Einstein Dr., Princeton, NJ 08540, USA}
\email{dallas.albritton@ias.edu}

\author[T. Barker]{Tobias Barker}
\address[T. Barker]{Department of Mathematical Sciences, University of Bath, Bath BA2 7AY, UK}
\email{tobiasbarker5@gmail.com}

\author[C. Prange]{Christophe Prange}
\address[C. Prange]{Cergy Paris Universit\'e, Laboratoire de Math\'ematiques AGM, UMR CNRS 8088, France}
\email{christophe.prange@cyu.fr}
\date{\today}

\begin{document}
\begin{abstract}
We establish a local-in-space short-time smoothing effect for the Navier-Stokes equations in the half space. The whole space analogue, due to Jia and {\v S}ver{\'a}k~\cite{jiasverakselfsim}, is a central tool in two of the authors' recent work on quantitative $L^3_x$ blow-up criteria~\cite{barker2020quantitative}. The main difficulty is that the non-local effects of the pressure in the half space are much stronger than in the whole space. As an application, we demonstrate that the critical $L^3_x$ norm must concentrate at scales $\sim \sqrt{T^* - t}$ in the presence of a Type~I blow-up. 
\end{abstract}

\maketitle

\tableofcontents

\parskip   2pt plus 0.5pt minus 0.5pt

\section{Introduction}

This paper is devoted to the study of local smoothing properties for the Navier-Stokes equations 
\begin{equation}\label{e.nse}
\tag{NS}
\left\lbrace
\begin{aligned}
\partial_t u-\Delta u+u\cdot\nabla u+\nabla p=0 \\
\div u = 0
\end{aligned}
\right.
\end{equation}
in the half space $\R^3_+$ with no-slip conditions on the boundary $\partial\R^3_+$.

Our results are motivated by the following general question: \emph{What initial data produce smooth solutions?} More specifically, we ask,
\begin{quote}
\emph{Is it possible to quantify smoothing effects in terms of local properties of the initial data?}
\end{quote}
This question has been subject to intense research in recent years for the three-dimensional Navier-Stokes equations in the whole space, in the wake of the seminal work~\cite{jiasverakselfsim}. In~\cite[Theorem 3.1]{jiasverakselfsim}, Jia and \v Sver\'ak established a \emph{local-in-space short-time smoothing effect}
 in the context of local energy solutions to~\eqref{e.nse} with data of locally uniformly bounded energy that in addition belongs locally to a subcritical space $L^m(B_1)$, $m>3$. That result was generalized by two groups, roughly at the same time and independently, to critical spaces: on the one hand by Kang, Miura, and Tsai~\cite{KMT21} for data locally in $L^3(B_1)$ by a compactness argument akin to~\cite{jiasverakselfsim}, and on the other hand by Barker and Prange~\cite{barker2018localized} for data locally in $L^3$, $L^{3,\infty}$, and $B^{-1+3/p}_{p,\infty}(B_1)$, $p\in(3,\infty)$, by a Caffarelli--Kohn--Nirenberg-type iteration~\cite{ckn}. In either case, the critical norm is assumed to be small in $B_1$. We also refer to the works~\cite{bradshaw2019global,KMT-arxiv20,KMT-arxiv21} along the same line of research, where further refinements are obtained.

Local-in-space short-time smoothing is a very useful and versatile tool for the study of the Navier-Stokes equations in the whole space. It was originally introduced in~\cite{jiasverakselfsim} to quantify the spatial asymptotics of the profile of (forward) self-similar solutions; this was key to their existence result for large-data self-similar solutions.\footnote{Later, existence was also established by more elementary methods, see~\cite{bradshawtsaiII}, based on a Calder{\'o}n-type splitting, and subsequent works. Forward self-similar solutions are intimately connected to the  non-uniqueness of Leray-Hopf solutions~\cite{jiasverakillposed,guillodsverak,albritton2021nonuniqueness}. The works~\cite{tsaidiscretely,KMT21} further apply localized smoothing to the spatial asymptotics of $\lambda$-discretely self-similar solutions with scaling factor $\lambda \approx 1$.} More recently, it was exploited to prove norm concentration results near potential singularities in the spirit of~\cite{barker2018localized}. Third, it is a key tool for the backward propagation of certain scale-invariant quantities in the strategy of~\cite{barker2020quantitative}; see also the related work~\cite{Tao19}, where quantitative estimates are obtained by backward propagation of Fourier-based scale-invariant quantities, and~\cite{Stan1,Stan2}. This allowed the second and third authors of the present paper to develop a quantitative version of Seregin's $L^3$ blow-up criterion~\cite{sereginl3}. This is one motivation for the present work.

One heuristic interpretation of local-in-space short-time smoothing is the following: \emph{The nonlocal effects of the pressure do not substantially hinder the parabolic smoothing properties of the Navier-Stokes equations, at least locally in space and time} (see \textbf{(S)} in~\cite[p. 234]{jiasverakselfsim}). It is not clear at first sight that such a property holds, even in $\R^3$. It is all the more difficult to prove such local smoothing properties in the half space, for a number of fundamental reasons, related to the fact that the nonlocal effects of the pressure are much stronger in $\R^3_+$ than in $\R^3$:

\begin{enumerate}[leftmargin=*]
\item The no-slip boundary condition for the velocity is responsible for strong velocity gradients near the boundary and vorticity creation. The vorticity itself does not satisfy a `standard' boundary condition, though a nonlinear and nonlocal boundary condition was derived in \cite{Mae13}. As a consequence, as of now there is still no proof of a Constantin and Fefferman-type geometric nonlinearity depletion criterion~\cite{CF93} outside the critical setting. In the critical case, i.e., under a Type I assumption, such a result was obtained in \cite{gigahsumaekawaplaner} thanks to the proof of a complicated Liouville theorem, and in \cite{barker2019scale} via a new strategy based on the stability of Type I singularities. The relationship between boundary effects and potential singularity formation is discussed in~\cite{lariostiti}.  

\item Contrary to the whole space, there are obstructions to spatial smoothing in the half space for suitable solutions. This was demonstrated by an example of Kang~\cite{KangUnbounded2005} and of Seregin and {\v S}ver{\'a}k~\cite{sereginsverakshear}: in the local setting, there are bounded flows with unbounded derivatives. Based on Kang's example, Kang, Lai, Lai and Tsai \cite{KLLT20,KLLT21} were able to construct a globally finite-energy solution to Navier-Stokes with a localized flux on the boundary and unbounded gradient everywhere on the boundary away from the flux.

\item In the half space the pressure can be decomposed into a Helmholtz part and a harmonic part. This decomposition is the key to the resolvent pressure formulas obtained in the paper \cite{MMPEnergy} inspired by the work \cite{DeschHieberPruss}. Koch and Solonnikov \cite{KS02} showed that there are examples of Stokes systems in the half space with divergence-form source term $\nabla\cdot F$ for which the harmonic pressure is not integrable in time. Hence in short, $p\not\simeq u\otimes u$ for the half space, contrary to the whole space. This issue brought about considerable difficulties in the work \cite{barker2019scale} when clarifying the relationship between different notions of Type I singularities in the half space; see also the discussion on page~\pageref{thm:localsmoothinglocal} above Theorem~\ref{thm:localsmoothinglocal}. A way to circumvent the difficulty was to rely on `fractional pressure estimates' pioneered in the work~\cite{CK18} and reproved in \cite{barker2019scale} via the formulas of~\cite{MMPEnergy}.

\item The pressure associated to the Stokes resolvent problem in a bounded domain $\Omega$ with no-slip boundary condition and source term $f$ satisfies the following bound:
\begin{equation}\label{e.pressres}
\|p\|_{L^2(\Omega)}\lesssim_\alpha |\lambda|^{-\alpha}\|f\|_{L^2(\Omega)},\qquad\mbox{for}\quad \alpha\in[0,1/4)
\end{equation}
as showed in \cite{NS03,TW20}. In \cite{Tolk20} the optimality of the threshold $1/4$ is established. Notice that the power $\alpha$ in estimate \eqref{e.pressres} breaks the natural scaling of the equations in the whole space, where $\alpha=1/2$. Let us emphasize though that in a bounded domain there is no such scale-invariance. The bound \eqref{e.pressres} turns, via Dunford's formula, into a short-time estimate for the pressure associated to the unsteady Stokes problem with a singularity $O(t^{-3/4-\delta})$ for $\delta>0$. This singularity is consistent with the short-time estimates obtained in \cite[Proposition 2.1]{MMPEnergy}; see Lemma~\ref{lem:localenergyestslemma} below. 
\end{enumerate}

In spite of these difficulties, the Navier-Stokes equations in the half space prove to have sufficiently good localization properties to be able to establish the following results. As in~\cite{jiasverakselfsim}, we prove local smoothing at the level of \emph{local energy solutions}. In the half space, these solutions were developed in~\cite{MMPEnergy}, see specifically Definition 1.1 in~\cite{MMPEnergy}.
Recall
\begin{equation}
	\norm{u_0}_{L^2_{\uloc}(\R^3_+)}^2 := \sup_{x_0 \in \overline{\R^3_+}} \int_{B(x_0) \cap \R^3_+} |u_0|^2 \, dx \, .
\end{equation}
We also assume the decay condition
\begin{equation}
\label{eq:decaycond}
	\lim_{\stackrel{|x_0| \to +\infty}{x_0 \in \overline{\R^3_+}}} \int_{B(x_0) \cap \R^3_+} |u_0|^2 \, dx = 0 \, .
\end{equation}

Let $x_0 \in \overline{\R^3_+}$ and $T > 0$. Define $\Omega_R(x_0) := B_R(x_0) \cap \R^3_+$.\footnote{For further notations, we refer to the last subsection of the `Introduction'.} 

\begin{theorem}[Localized smoothing, global setting]
\label{thm:localsmoothinghalfspace}
Let $u$ be a local energy solution on $\R^3_+ \times (0,T)$ with initial data $u_0$ satisfying $\norm{u_0}_{L^2_{\uloc}(\R^3_+)} \leq M$ and the decay condition~\eqref{eq:decaycond}.

Let $\norm{u_0}_{L^m(\Omega_3(x_0))} \leq N$ with $m \in [3,+\infty)$. If $m=3$, we further require the smallness condition $N \leq N_0 \ll 1$.

Then there exists $S = S(M,N,m) \in (0,1]$ satisfying the following property: For $\bar{S} = \min(S,T)$, we have
\begin{equation}
	\label{eq:linfitysmoothingu}
	\sup_{t \in (0,\bar{S})} t^{\frac{3}{2m}} \norm{u(\cdot,t)}_{L^\infty(\Omega_1(x_0))} \les_m N + N^{C_m}
\end{equation}
for a constant $C_m \geq 1$. 
\end{theorem}

\begin{remark}[Subcritical refinements]
	\label{rmk:subcriticalrefinements}
\textit{Under the hypotheses of Theorem~\ref{thm:localsmoothinghalfspace}, if additionally $m > 3$, then the following hold:
\begin{itemize}[leftmargin=*]
	\item For all $p \in [m,+\infty]$, we have
\begin{equation}
	\label{eq:refinedlpest}
	\sup_{t \in (0,\bar{S})} t^{\frac{3}{2}\left( \frac{1}{m} - \frac{1}{p} \right)} \norm{u(\cdot,t)}_{L^p(\Omega_1(x_0))} \les_m N + N^{C_m} \, .
\end{equation}
 	\item When $M, N \geq C_{\rm univ} > 0$, we have
\begin{equation}
S = O(1) M^{-O(1)} N^{-O(1)} \, .
\end{equation}
\end{itemize}
}
\end{remark}

The property~\eqref{eq:refinedlpest} is a consequence of the decomposition~\eqref{eq:muhdecomposition} of $u$ into a strong solution $a$ with $L^m$ initial data and a H{\"o}lder-continuous remainder $v$.

\smallskip

Our main application is the following concentration result, analogous to~\cite[Theorem 2]{barker2018localized}.

\begin{theorem}[Global concentration]
	\label{thm:concentration}
	Let $T^* > 0$. Let $u$ be a local energy solution on $\R^3_+ \times (0,T^*)$. Suppose that $u$ is locally bounded on $\overline{\R^3_+} \times [0,T^*)$ and that $(x^*,T^*)$ is a singular point of $u$, where $x^* \in \overline{\R^3_+}$.

	Moreover, suppose that the singularity is Type~I in the following global sense: for some $r_0 \in (0,\sqrt{T^*}]$,
	\begin{equation}
		\label{eq:muhglobaltypeibound}
		\sup_{x \in \overline{\R^3_+}} \sup_{0  < r \leq r_0} r^{-\frac{1}{2}} \| u \|_{L^\infty_t L^2_x(\Omega_r(x) \times (T^*-r^2,T^*))} \leq M.
	\end{equation}

	 Then there exists $\bar{t} = \bar{t}(T^*,M,r_0) \in [T^* - r_0^2,T^*)$ such that for all $t \in [\bar{t},T^*)$, we have the following concentration of the critical $L^3$ norm:
	\begin{equation}
		\label{eq:muhconcentrationineq}
		\| u(\cdot,t) \|_{L^3(\Omega_{R(t)}(x^*))} > N_0,
	\end{equation}
	where
	\begin{equation}
	R(t) := 3 \times \sqrt{\frac{T^*-t}{S(M)}},
	\end{equation}
	and $S(M) = S(M,N_0,3)$ and $N_0$ are as in Theorem~\ref{thm:localsmoothinghalfspace}.
\end{theorem}

The proof is an immediate consequence of Theorem~\ref{thm:localsmoothinghalfspace} and a rescaling procedure, so we summarize it here:

\begin{proof}[Proof of Theorem~\ref{thm:concentration}]
Define $\bar{t} = T^* - S(M) r_0^2$. As in~\cite[Section 4.2]{barker2018localized}, we prove the contrapositive. Suppose that, for some $t \in [\bar{t},T^*)$, the concentration inequality~\eqref{eq:muhconcentrationineq} is violated. We time-translate and rescale (under the Navier-Stokes scaling symmetry) the solution $u$ so that the time interval $(t,T^*)$ becomes $(0,S)$ in the new variables. Let $\tilde{x}^*$ be the image of $x^*$ under this transformation. Thanks to the global Type~I bound~\eqref{eq:muhglobaltypeibound} and the definition of $\bar{t}$, the new solution $\tilde{u}$ satisfies $\| \tilde{u}(\cdot,0) \|_{L^2_{\rm uloc}(\R^3_+)} \leq M$. Moreover, by the violation of~\eqref{eq:muhconcentrationineq} and the definition of $R(t)$, we have the smallness of the critical $L^3$ norm in $\Omega_3(\tilde{x}^*)$: $\| \tilde{u}(\cdot,0) \|_{L^3(\Omega_3(\tilde{x}^*))} \leq N_0$. Hence, $\tilde{u}$ satisfies the assumptions of Theorem~\ref{thm:localsmoothinghalfspace}. We conclude that  $\tilde{u}$ is bounded in a parabolic neighborhood of $(\tilde{x}^*,S)$, so $u$ is bounded in a parabolic neighborhood of $(x^*,T^*)$. \end{proof}

The most classical notion of Type~I, corresponding to the ODE blow-up rate in semilinear heat equations, is in terms of the $\sup$-norm: $\| u(\cdot,t) \|_{L^\infty(\R^3_+)} \sim (T^* - t)^{-1/2}$. On the other hand, this terminology occasionally refers to boundedness of \emph{some} scaling-invariant quantity at the blow-up time, for example, $\| u(\cdot,t) \|_{L^{3,\infty}(\R^3_+)} \les 1$. The assumption~\eqref{eq:muhglobaltypeibound} is a weak form of Type~I, adapted to the scaling-invariant energy. For Leray-Hopf solutions,~\eqref{eq:muhglobaltypeibound} is \emph{implied} by the $L^\infty$ and $L^{3,\infty}$ conditions, though this is not trivial, see~\cite[Theorem 2]{barker2019scale} for the precise statement with boundary and~\cite{albrittonbarkerlocalregI} for a discussion without boundary.

\smallskip

It is also possible to prove \emph{local} versions of the above theorems wherein the global background assumption $\| u_0 \|_{L^2_{\rm uloc}(\R^3_+)} \leq M$ is replaced by a local assumption on the solution $u$ itself.

\begin{theorem}[Localized smoothing, local setting]
\label{thm:localsmoothinglocal}
Let $(u,p)$ be a Navier-Stokes solution on $\Omega_3(x_0) \times (0,T)$ satisfying
\begin{equation}
	\| u \|_{L^\infty_t L^2_x(\Omega_3(x_0) \times (0,T))} + \| \nabla u \|_{L^2_{t,x}(\Omega_3(x_0) \times (0,T))} + \| p \|_{L^{\zeta_t}_t L^{\zeta_x}_x(\Omega_3(x_0) \times (0,T))} \leq M \, ,
\end{equation}
where $(\zeta_x,\zeta_t)$ is a fixed pair of exponents specified in Section~2.\footnote{This ensures, among other things, that $up \in L^1(\Omega_3(x_0) \times (0,T))$.}

Suppose that $u(\cdot,t)$ vanishes on $\p \Omega_3(x_0) \cap \p \R^3_+$ for all $t \in (0,T)$ and that $(u,p)$ satisfies the local energy inequality.

Let $\norm{u_0}_{L^m(\Omega_3(x_0))} \leq N$ with $m \in [3,+\infty)$. If $m=3$, we further require the smallness condition $N \leq N_0 \ll 1$. Suppose that $u(\cdot,t) \overset{t \to 0^+}{\longrightarrow} u_0$ in $L^2(\Omega_3(x_0))$. 

Then there exists $S = S(M,N,n) \in (0,1]$ satisfying the following property: for $\bar{S} = \min(S,T)$, we have
\begin{equation}
	\sup_{t \in (0,\bar{S})} t^{\frac{3}{2m}} \norm{u(\cdot,t)}_{L^\infty(\Omega_1(x_0))} \les_m N + N^{C_m}
\end{equation}
for a constant $C_m \geq 1$.
\end{theorem}

This kind of refinement first appeared in~\cite[Theorem 1.1]{KMT21} without boundary.

As a consequence of Theorem~\ref{thm:localsmoothinglocal}, we have the following localized concentration theorem. We focus on near-boundary concentration below. To our knowledge, this type of localized concentration theorem is new even in the absence of boundary.

\begin{theorem}[Localized concentration]
	\label{thm:localizedconcentration}
	Let $(u,p)$ be a suitable weak solution on $Q_4^+$, in the sense of Definition~\ref{def.sws}, satisfying
\begin{equation}\label{boundedenergyonescalethm}
 \| \nabla u \|_{L^2_{t,x}(Q_4^+)} + \| p \|_{L^{\frac{3}{2}}_{t,x}(Q_4^+)} \leq M_{0}
\end{equation}
and
\begin{equation}
    \label{boundedkineticenergyallscalesthm}
    \sup_{(y_0,s_{0})\in Q^+_3}\sup_{0<r\leq 1}{r^{-\frac{1}{2}}}\|u\|_{L^{\infty}_{t}L^{2}_{x}(\Omega_r(y_0) \times (s_0-r^2,s_0))}\leq A_{0}.
\end{equation}
Suppose moreover that $u$ is locally bounded on $\overline{B^{+}_4}\times (-1,0)$ and that $(x^*,0)\in  \overline{B^{+}} \times \{0\}$ is a singular point of $u$.

 Then the above assumptions imply that there exists $M(M_0,A_0)>0$, $\bar{t}(M)\in (-16,0)$ and $S(M) \in (0,1]$ such that for all $M_0$ and $A_{0}$ sufficiently large\footnote{This means that there exists a universal constant $N_{\rm univ}\in [1,\infty)$ such that for all $M_{0}\geq N_{\rm univ}$ and $A_{0}\geq N_{\rm univ}$ we have the result. } the following holds true. For every $t\in [\bar{t}(M),0)$, we have the following concentration of the critical $L^{3}$ norm:
 \begin{equation}
		\label{eq:muhconcentrationineqlocal}
		\| u(\cdot,t) \|_{L^3(\Omega_{R(t)}(x^*))} > N_0,
	\end{equation}
	where
	\begin{equation}\label{concentrationradius}
	R(t) := 3 \times \sqrt{\frac{-t}{S(M)}}.
	\end{equation}
	 In the above, $M(M_0,A_0)$ is as in Proposition \ref{scaleinvarestnotcentre}. Furthermore, $S(M) = S(M,N_0,3)$ and $N_0$ are as in Theorem~\ref{thm:localsmoothinglocal}.
\end{theorem}

\subsection*{Strategy of the proof}

We begin by explaining the subcritical case $m>3$. The general strategy is as in~\cite{jiasverakselfsim}. We decompose the solution $u$ as
\begin{equation}
	\label{eq:muhdecomposition}
u = a + v,
\end{equation}
where $a$ is a strong solution of the Navier-Stokes equations with (sub)critical initial data $a_0$ satisfying $a_0 \equiv u_0$ in $B_2$ and $a_0 \equiv 0$ outside $B_3$. The remainder $v$ satisfies a perturbed Navier-Stokes equation~\eqref{eq:perturbednse}, which has lower order terms, and initial data $v_0$ satisfying $v_0|_{B_2} \equiv 0$. We then develop an $\varepsilon$-regularity criterion for the perturbed equation up to the initial time in order to establish that $v$ is H{\"o}lder continuous up to the initial time.\footnote{In general, it is an interesting observation that the special structure of the energy inequality allows one to localize the solution under a supercritical background assumption, namely, control on the energy. Na{\"i}vely, to localize a solution to a nonlinear partial differential equation (PDE) without any such structure, one would require that the solution belongs to a critical space.} Recall that in the whole space, an $\varepsilon$-regularity criterion up to the initial time will `kick in' on short times because the energy controls $\| u \|_{L^{10/3}_{t,x}}$ whereas $\varepsilon$-regularity requires smallness of $\| u \|_{L^{3}_{t,x}}$.\footnote{In four dimensions, although $L^3_{t,x}$ is `critical for the energy inequality,' one can imagine a va\-riation on localized smoothing under the additional assumption $M \ll 1$.}

It is standard that  $\varepsilon$-regularity also depends on the pressure $p$, and a typical choice of quantity is $\| p \|_{L^{3/2}}$.\footnote{In the interior case, there is a way to introduce a `local pressure', due to Wolf~\cite{wolflocalpressure}, see also Kwon~\cite[Theorem 1.6]{kwon2021role}, which contains an application to interior localized smoothing in a bounded domain. We expect that these arguments cannot be extended to the pressure at the boundary.} These choices are convenient for treating the term $\int up \, dx \,dt$ in the local energy inequality. This brings us to our first difficulty, namely, that for solutions of the linear Stokes equations in the half space with initial data in $L^2$, the pressure estimates are only known in $L^{4/3-}_{t,\loc} L^2_x$. Notice the low time integrability. We already mentioned this in connection with~\eqref{e.pressres}. Therefore, we must prove $\varepsilon$-regularity for the perturbed system~\eqref{eq:perturbednse} under a new assumption on the pressure. Essentially, we require smallness of $p \in L^{1+\delta_t}_t L^{2-\delta_x}_x$ with $0 < \delta_x \ll \delta_t \ll 1$. Our velocity assumption, which has the H{\"o}lder conjugate exponents, compensates for the low time integrability but remains controlled, with room to spare, by the energy space $L^\infty_t L^2_x \cap L^2_t H^1_x$. Our precise assumptions are discussed in~\eqref{eq:zetadef}. Our subcritical $\varepsilon$-regularity criterion is Proposition~\ref{pro:epsilonreg}.

Concerning implementation of the $\varepsilon$-regularity criterion, there are at least three strategies:
\begin{itemize}[leftmargin=*]
\item \emph{Direct iterative method of Caffarelli--Kohn--Nirenberg}~\cite{ckn}. Variations were exploited by Barker-Prange~\cite{barker2018localized} with critical lower order terms and by Dong-Gu~\cite{dongguboundarypartial} with boundary and Bian Wu~\cite{wu2021partially} in dimension $n=4$
\item \emph{Lin's compactness method}~\cite{lin}, see also Seregin-Ladyzhenskaya~\cite{ladyzhenskayaseregin}. This was adapted by Jia-Sverak~\cite{jiasverakselfsim} and Kang-Miura-Tsai~\cite{KMT21} to accomodate (sub)critical lower order terms.
\item \emph{De Giorgi-type method}, as implemented by Vasseur~\cite{vasseurpartialreg} (see also Wang-Wu~\cite{wangwuunified} in dimension $n=4$).
\end{itemize}
Our approach is to adapt Lin's compactness method to the new quantities. We describe the main idea in Section~\ref{sec:epsilonreg}, in a way which we hope is suitable for newcomers, once we have defined the necessary quantities $Y(R)$ and $Y_{\rm osc}(R)$. For now, we mention that Lin's method is actually comprised of two compactness arguments: one for $Y(R)$ and one for $Y_{\rm osc}(R)$.

There is a second difficulty, which concerns only the critical case $m=3$ and was also encountered in~\cite{barker2018localized,KMT21}. In this case, the first compactness argument in Lin's method still yields a subcritical Morrey bound `just below' $L^\infty$ in terms of $Y(R)$, but the second compactness argument fails to improve the decay of $Y_{\rm osc}$, the step which would yield H{\"o}lder continuity. This is because solutions of the limiting system~\eqref{eq:limitingstokessystem} are not guaranteed to be bounded, let alone H{\"o}lder continuous, when $m=3$. In~\cite{barker2018localized}, Barker and Prange overcome this difficulty by using parabolic regularity theory to bootstrap the regularity of the perturbation $v$, see~\eqref{eq:muhdecomposition}, from subcritical Morrey to $C^\alpha$. In principle, this is also possible here, but it is not necessary for our application. Rather, following~\cite{KMT21}, we combine (i) the subcritical Morrey estimates for the perturbation~$v$, (ii) the critical estimates for~$a$, and (iii) the standard $\varepsilon$-regularity criterion (without lower order terms) for~$u$ to conclude the $L^\infty$-smoothing~\eqref{eq:linfitysmoothingu}. Unlike~\cite{barker2018localized}, we do not demonstrate that~$v$ is H{\"o}lder continuous when $m=3$.

Finally, a small novelty of our approach, compared to~\cite{KMT21}, is that we build smallness of $a$ into our compactness arguments, which completely bypasses the estimates in~\cite[Section~4]{KMT21}.

\subsection*{Notation}

Let $z_0 = (x_0,t_0) \in \R^{3+1}$ and $R > 0$. We define the ball $B_R(x_0) := \{ x \in \R^3 : |x - x_0| < R \}$ and parabolic ball $Q_R(z_0) := B_R(x_0) \times (t_0-R^2,t_0)$.

Recall that $\R^3_+ := \{ (x_1,x_2,x_3) \in \R^3 : x_3 > 0 \}$. If $x_0 \in \p \R^3_+$, we define $B^+_R(x_0) := B_R(x_0) \cap \R^3_+$ and $Q_R^+(z_0) := B_R^+(x_0) \times (t_0-R^2,t_0)$.

We define $\p_{\rm flat} B^+_R(x_0) := B_R(x_0) \cap \{ d(x) = d(x_0) \}$, where $d(x) := x \cdot e_3$ is the signed distance to $\p \R^3_+$. Similarly, $\p_{\rm flat} Q^+_R(z_0) := \p_{\rm flat} B^+_R(x_0) \times (t_0 - R^2,t_0)$.

To simultaneously treat the interior and boundary scenarios, we introduce the notation $B^\iota_R(x_0)$ and $Q^{\iota}_R(z_0)$, where $\iota \in \{ {\rm int}, {\rm bd} \}$.
Here, $B^{\rm int}_R(x_0) = B_R(x_0)$ and $B^{\rm bd}_R(x_0) = B^+_R(x_0)$, with an analogous convention for parabolic balls $Q^\iota_R(z_0)$.

When $x_0 = 0$ (resp. $z_0 = 0$, $R=1$), we may omit $x_0$ (resp. $z_0$, $R$) from the above notation. As defined above, $\Omega_R(x_0) := B_R(x_0) \cap \R^3_+$.

Often we do not distinguish between scalar- and vector-valued function spaces in notation.

For vectors $u$ and $v$, we write $(u \otimes v)_{ij} = u_i v_j$. For matrices $F$ and $G$, we write $F : G = \sum_{ij} F_{ij} G_{ij}$. For matrix-valued $F$, we write $(\div F)_i = \sum_j \p_j F_{ij}$.

\section{Preliminaries}

We define\footnote{More generally, we could define scaling exponents for a variety of homogeneous function spaces to keep track of viable embeddings and interpolations. For example, $\# L^q_t L^p_x = \#(p,q)$.}
the \emph{scaling exponent} $\#(p,q)$, whenever $(p,q) \in [1,\infty]^2$, by
\begin{equation}
	\#(p,q) = - \frac{3}{p} - \frac{2}{q}.
\end{equation}
If $f \in L^1_\loc(\R^{3+1})$, we have
\begin{equation}
	\norm{f(\lambda \cdot)}_{L^q_t L^p_x(\R^{3+1})} = \lambda^{\#(p,q)} \norm{f}_{L^q_t L^p_x(\R^{3+1})}.
\end{equation}

\begin{lemma}[Local energy and pressure estimates]
	\label{lem:localenergyestslemma}
Let $M\geq 1$. Under the assumptions of Theorem~\ref{thm:localsmoothinghalfspace}, there exists a time
\begin{equation}
    \label{eq:lowerboundonexistence}
	S_1 \ges \min(1,M^{-96})
\end{equation}
 satisfying the following property. Let $\bar{S}_1 = \min(S_1,T)$. Then
\begin{equation}
	\label{eq:localenergyests}
	\sup_{x_0 \in \R^3_+}\left(\sup_{t \in (0,\bar{S}_1)} \int_{\Omega_2(x_0)} |u(x,t)|^2 \, dx + \iint_{\Omega_2(x_0) \times (0,\bar{S}_1)} |\nabla u|^2 \, dx \, ds\right) \les M^2
\end{equation}
and $p$ satisfies 
\begin{equation}
	\label{eq:pressureuotimesu}
	\sup_{x_0 \in \R^3_+} \norm{p - [p]_{\Omega_2(x_0)}}_{L^{\frac43,\infty}_t L^2_x(\Omega_2(x_0) \times (0,\bar{S}_1))} \les 
	M^2.
\end{equation}
\end{lemma}
In particular, it follows from \eqref{eq:localenergyests} that whenever $\#(p,q) = -3/2$ and $q \in [2,+\infty]$,
\begin{equation}
	\sup_{x_0 \in \R^3_+} \norm{u}_{L^q_t L^p_x(\Omega_2(x_0)  \times (0,\bar{S}_1))} \les M.
\end{equation}
The velocity estimate is a restatement of Corollary 5.9 and Proposition 5.7 in\  \cite{MMPEnergy}.\footnote{See specifically the equation below (5.16) in that paper for the lower bound~\eqref{eq:lowerboundonexistence} on the time $S_1$. Notice that the notation $M$ is different therein and may be fixed to $M=2$.} The pressure 
estimates are described in \cite[Propositions 2.1--2.3]{MMPEnergy}. 

In the following, we write $\mathcal{Q}^\iota = Q^\iota_R(z_0)$. We refer to the subsection `Notation' of the `Introduction' where these parabolic cylinders and further notations are defined.

Let $a \in L^5_{t,x}(\mathcal{Q}^\iota)$.
Consider the following \emph{perturbed Navier-Stokes equations}:
\begin{equation}
\label{eq:perturbednse}
\tag{${\rm NS}_a$}
\left\lbrace
\begin{aligned}
	\p_t v - \Delta v + (v + a) \cdot \nabla v + \div (a \otimes v) + \nabla q &= 0 \\
	\div v &= 0 \, .
	\end{aligned}
	\right.
\end{equation}

\begin{definition}\label{def.sws}
We say that $(v,q)$ is a \emph{suitable weak solution} of~\eqref{eq:perturbednse} in $\mathcal{Q}^\iota$ if the following criteria are met:
\begin{enumerate}[leftmargin=*]
\item $v \in L^\infty_t L^2_x \cap L^2_t H^1_x(\mathcal{Q}^\iota)$, $q \in L^1(\mathcal{Q}^\iota)$, and $vq \in L^1(\mathcal{Q}^\iota)$. If $\iota = {\rm bd}$, then also $v|_{\p_{\rm flat} \mathcal{Q}^\iota} = 0$.
\item \eqref{eq:perturbednse} is satisfied in $\mathcal{Q}^\iota$ in the sense of distributions.
\item (Weak continuity in time) For all $\varphi \in L^2(B^\iota_R(x_0))$,
\begin{equation}
	\Big( t \mapsto \int_{B^\iota_R(x_0)} v(x,t) \varphi \, dx \Big) \in C(t_0 + [-R^2,0]).
\end{equation}
\item (Local energy inequality) For all non-negative $\Phi \in C^\infty_0(B_R(z_0) \times (t_0-R^2,+\infty))$ and $t \in  (t_0-R^2,t_0)$,
\begin{equation}
\label{eq:localenergyineq}
\begin{aligned}
	&\int_{B^\iota_R(x_0)} |v(x,t)|^2 \Phi(x,t) \, dx +2 \iint_{B^\iota_R(x_0) \times (-\infty,t)} |\nabla v|^2 \Phi \, dx \, ds  \\
	&\quad \leq \iint_{B^\iota_R(x_0) \times (-\infty,t)} |v|^2 (\p_t + \Delta) \Phi  \, dx \, ds + \iint_{B^\iota_R(x_0) \times (-\infty,t)} (|v|^2 + 2 q) v \cdot \nabla \Phi \, dx \, ds \\
	&\quad\quad + \iint_{B^\iota_R(x_0) \times (-\infty,t)} |v|^2 a \cdot \nabla \Phi \, dx \, ds\\
	&\quad\quad +  2\iint_{B^\iota_R(x_0) \times (-\infty,t)} a \otimes v : (\Phi \nabla  v  + v \otimes \nabla \Phi) \, dx \, ds.
	\end{aligned}
\end{equation}
\end{enumerate}
\end{definition}

Let $\zeta = (\zeta_x,\zeta_t)$ satisfying
\begin{equation}
	\label{eq:zetadef}
	\zeta_x = 2-\delta_x, \quad \zeta_t = 1+\delta_t, \quad \#\zeta = -\frac{7}{2} + \delta_0,
\end{equation}
 where $0 < \delta_x, \delta_t, \delta_0 \ll 1$. These are the exponents for the pressure $q$. Let $\xi = (\xi_x, \xi_t)$ be its H{\"o}lder conjugate. Hence,
 \begin{equation}
	\#\xi = -\frac{3}{2} - \delta_0.
 \end{equation}
 These are exponents for the velocity $v$. In principle, the constants below may depend on $\zeta$ and $\xi$. We consider them to be fixed in the following analysis.

 We also require exponents for the perturbation $a$. Let $m \in [3,+\infty)$ and let $\#(p_1,q_1) = -3/2$ satisfying $q_1 \in (2,+\infty)$ and $1/\zeta_t = 1/2+1/q_1$ (in particular, heuristically, $q_1 = 2^+$ and $p_1 = 6^-$). Let $\#(p_2,q_2) = -5/2$ satisfying $q_2 = 2-\delta_2$ where $0 < \delta_2 \ll 1$ (in particular, $p_2=2^+$). Similarly to $\zeta$ and $\xi$, we consider $p_1,q_1,p_2,q_2$ as fixed and suppress dependence on them in the notation $\les$. We write
 \begin{equation}\label{e.norma}
 \begin{aligned}
	&\norm{a}_{\mathbf{A}_m(\mathcal{Q}^\iota)} = R^{1-\frac{3}{m}}\norm{a}_{L^{5m/3}_{t,x}(\mathcal{Q}^\iota)} + \\
	&\quad R^{\#(p_1,q_1)+1}\norm{a}_{L^{q_1}_t L^{p_1}_x(\mathcal{Q}^\iota)} + R^{\#(p_2,q_2)+2}\norm{\nabla a}_{L^{q_2}_t L^{p_2}_x(\mathcal{Q}^\iota)} \, .
	\end{aligned}
 \end{equation}
 Notice that $\#(5m/3,5m/3) = -3/m$.
 We also define the Morrey-type space
 \begin{equation}
	\norm{a}_{\mathbf{M}_m(Q_R(z_0))} = \sup_{\bar{Q}} \norm{a}_{\mathbf{A}_m(\bar{Q})}
 \end{equation}
 where $\bar{Q}$ ranges over parabolic balls $Q_r(z_1) \subset Q_R(z_0)$. Similarly,
 \begin{equation}
	\norm{a}_{\mathbf{M}_m(Q^+_R(z_0))} = \sup_{\bar{Q}} \norm{a}_{\mathbf{A}_m(\bar{Q})} + \sup_{\bar{Q}^+} \norm{a}_{\mathbf{A}_m(\bar{Q}^+)},
 \end{equation}
 where $\bar{Q}$ ranges over parabolic balls $Q_r(z_1) \subset Q^+_R(z_0)$
 and $\bar{Q}^+$ ranges over parabolic half-balls $Q_r^+(z_1) \subset Q^+_R(z_0)$ with $z_0 \in \p_{\rm flat} Q^+$. In particular, $\norm{a}_{\mathbf{A}_m(Q^\iota_R(z_0))} \leq \norm{a}_{\mathbf{M}_m(Q^\iota_R(z_0))}$ and $\norm{a}_{\mathbf{M}_m(Q^\iota_r(z_0))} \leq \norm{a}_{\mathbf{M}_m(Q^\iota_R(z_0))}$ when $r \in (0,R)$.

 Finally, we demonstrate that, in the setting of the proof of Theorem~\ref{thm:concentration}, the coefficient $a$ will belong to the desired Morrey spaces.

 \begin{lemma}
    \label{lem:acontrolledinMm}
     Let $m \in [3,+\infty)$, $u_0 \in L^m_\sigma(\R^3_+)$ with $\| u_0 \|_{L^m(\R^3_+)} \leq N$, and $a$ be the strong solution on $\R^3_+ \times (0,T_m(u_0))$ constructed in Proposition~\ref{pro:Lmsoltheory}. Consider $a$ as extended backward-in-time by zero. Then, for all $t_0 \in (0,T_m)$ and interior and boundary balls $Q^\iota_2(z_0)$, where $\iota \in \{ {\rm int}, {\rm bd} \}$ and $z_0 = (x_0,t_0)$, we have
     \begin{equation}
        \| a \|_{\mathbf{M}_m(Q^\iota_2(z_0))} \les_m N.
     \end{equation}
 \end{lemma}
 \begin{proof}
    Recall from~\eqref{e.norma} that $\| \cdot \|_{\mathbf{A}_m}$ in the definition of $\|\cdot\|_{\mathbf{M}_m}$ is comprised of three parts. The $L^{5m/3}_{t,x}$ part is controlled by~\eqref{eq:linearestimatelebesguem}. Regarding the remaining two parts, we observe from~\eqref{eq:linearestimatem} and~\eqref{eq:linearderivestm} that
    \begin{equation}
        \| a \|_{L^{l_1,\infty}_t L^{s_1}_x(\R^3_+ \times (0,T_m))} + \| \nabla a \|_{L^{l_2,\infty}_t L^{s_2}_x(\R^3_+ \times (0,T_m))} \les_m N
    \end{equation}
    whenever $s_1, s_2 \geq m$, $\#(s_1,l_1) = -3/m$, and $\#(s_2,l_2) = -1-3/m$. Applying H{\"o}lder's inequality for Lorentz spaces in the time variable yields the desired estimate when $s_1 = p_1$ and $s_2 = p_2$.
 \end{proof}

 \section{$\varepsilon$-regularity for perturbed Navier-Stokes system}
 \label{sec:epsilonreg}

In this section, all constants are allowed to depend on $m$ unless specified otherwise.

\begin{proposition}[Subcritical $\varepsilon$--regularity]
\label{pro:epsilonreg}
Let $m > 3$.
There exist $\varepsilon_{\rm CKN}, \bar{\alpha} > 0$ satisfying the following property. 
If $(v,q)$ is a suitable weak solution of the perturbed Navier-Stokes equations~\eqref{eq:perturbednse} on $Q^\iota$ in the sense of Definition \ref{def.sws} sa\-tisfying
\begin{equation}
	\norm{v}_{L^{\xi_t}_t L^{\xi_x}_x(Q^\iota)} + \norm{q}_{L^{\zeta_t}_t L^{\zeta_x}_x(Q^\iota)} + \norm{a}_{\mathbf{M}_m(Q^\iota)} \leq \varepsilon_{\rm CKN},
\end{equation}
then $v \in C^{\bar{\alpha}}_{\rm par}(Q_{1/2}^\iota)$ and
\begin{equation}
	\norm{v}_{C^{\bar{\alpha}}_{\rm par}(Q_{1/2}^\iota)}\les \norm{v}_{L^{\xi_t}_t L^{\xi_x}_x(Q^\iota)} + \norm{q}_{L^{\zeta_t}_t L^{\zeta_x}_x(Q^\iota)}.
\end{equation}
\end{proposition}

In the Navier-Stokes context, the idea of $\varepsilon$-regularity is that, if a particular local quantity is $O(\varepsilon)$, then the non-linearity is $O(\varepsilon^2)$, and the solution should enjoy the local regularity of solutions to the Stokes equations. This may be regarded as a local perturbation theorem around the Stokes equations. While the quantity in Proposition~\ref{pro:epsilonreg} may look arbitrary to a newcomer, the point is that it controls the local energy, due to the local energy inequality~\eqref{eq:localenergyineq}.

We now introduce some notation. If $(v,q)$ is a suitable weak solution on $Q_R^\iota$, we define
\begin{equation}
	Y^\iota(R,v,q) = R^{\#\xi} \norm{v}_{L^{\xi_t}_t L^{\xi_x}_x(Q_R^\iota)} + R^{\#\zeta+1} \norm{q - [q]_{B_R^\iota}}_{L^{\zeta_t}_t L^{\zeta_x}_x(Q_R^\iota)}
\end{equation}
where $[q]_{B_R^\iota} = \barint_{B_R^\iota} q \, dx$, and
\begin{equation}
	Y_{\osc}^\iota(R,v,q) = R^{\#\xi} \norm{v - (v)_{Q_R^\iota}}_{L^{\xi_t}_t L^{\xi_x}_x(Q_R^\iota)} + R^{\#\zeta+1} \norm{q - [q]_{B_R^\iota}}_{L^{\zeta_t}_t L^{\zeta_x}_x(Q_R^\iota)}
\end{equation}
where $(v)_{Q_R^\iota} = \bariint_{Q_R^\iota} v \, dx \, dt$. Note that $R Y^\iota(R,v,q)$ and $RY^\iota_\osc(R,v,q)$, rather than $Y^{\iota}$ and $Y_\osc^\iota$ themselves, have the critical scaling. The velocity term in $Y^\iota(R,v,q)$ has the same scaling as $L^\infty$. 

 To simplify, if the context is clear, we sometimes write $Y(R)$ instead of $Y(R,v,q)$, and similarly for the quantities $Y_\osc(R)$, $Y^+(R)$, and $Y^+_\osc(R)$, etc. We sometimes also consider the quantities centered at $z_0$ and write $Y(z_0,R,v,q)$, etc.

With this notation, we can state the $\varepsilon$-regularity with critical lower order terms. 

\begin{proposition}[Critical $\varepsilon$--regularity]
\label{pro:epsilonregcritical}
There exist $\varepsilon_{\rm CKN}, \bar{\alpha} > 0$ satisfying the following property. Suppose that $(v,q)$ is a suitable weak solution of the perturbed Navier-Stokes equations~\eqref{eq:perturbednse} on $Q^\iota$ in the sense of Definition~\ref{def.sws} satisfying
\begin{equation}
	\norm{v}_{L^{\xi_t}_t L^{\xi_x}_x(Q^\iota)} + \norm{q}_{L^{\zeta_t}_t L^{\zeta_x}_x(Q^\iota)} + \norm{a}_{\mathbf{M}_3(Q^\iota)} \leq \varepsilon_{\rm CKN}.
\end{equation}
Then, for the interior case, for all $\alpha \in [-1,0)$, we have the \emph{subcritical Morrey estimate}
\begin{equation}
	\sup_{z_0,R} R^{-\alpha} Y(z_0,R,v,q)  \les_\alpha \norm{v}_{L^{\xi_t}_t L^{\xi_x}_x(Q^\iota)} + \norm{q}_{L^{\zeta_t}_t L^{\zeta_x}_x(Q^\iota)},
\end{equation}
where $z_0 = (x_0,t_0)$, $|x_0| < 1/2$, and $R<1/2$ satisfy $Q(z_0,R) \subset Q^\iota$. Particular to the boundary case $\iota = {\rm bd}$, we additionally have that for all $\alpha\in(0,\bar\alpha)$,
\begin{equation}
	\sup_{z_0,R} R^{-\alpha} Y^+(z_0,R,v,q)  \les_\alpha \norm{v}_{L^{\xi_t}_t L^{\xi_x}_x(Q^+)} + \norm{q}_{L^{\zeta_t}_t L^{\zeta_x}_x(Q^+)},
\end{equation}
where $|x_0| < 1/2$, $d(x_0) = 0$, $R < 1/2$, and $Q^+(z_0,R) \subset Q^+$.
\end{proposition}

\subsection*{Summary of the method}
Lin's compactness method is in two steps:

\emph{Step 1. Morrey estimate}, or, \emph{improve the growth of $Y(R)$}.\footnote{Notice that the quantity $Y(R)$ should not actually \emph{decay} as $R \to 0^+$ unless the solution vanishes.} The goal of this step is to demonstrate that, for all $\alpha \in [-1,0)$, smallness of $Y(1)$ and the coefficient $a$ implies
\begin{equation}
    \label{eq:themorreyest}
Y(R) \les_\alpha R^{\alpha} Y(1) \text{ for all } R \leq 1.
\end{equation}
Notice that the smaller $\alpha$ is, the closer $v$ is to boundedness. We do not need to prove this for all $R \leq 1$ at once. Rather, we can show an improvement over a single scale, from $R=1$ to $R = \theta_0$, see Lemma~\ref{lem:interiormorrey}. Afterward, one can iterate the one-scale improvement to achieve the improvement for all scales, see Lemma~\ref{lem:iteratedestimates}.

To identify the good scale $\theta_0$, Lin employed a compactness/contradiction argument: If we have a sequence of solutions with $Y(1,v^{(k)},q^{(k)}) \leq \varepsilon_k \to 0^+$ (which also violate the conclusion), then the non-linearity is $O(\varepsilon_k^2)$. Since the coefficients $a^{(k)}$ are $O(\varepsilon_k)$, the lower order terms are also $O(\varepsilon_k^2)$. The normalized solutions $w^{(k)} = v^{(k)}/\varepsilon_k$ solve a Navier--Stokes-type equation whose non-linearity and lower-order terms are $O(\varepsilon_k)$. Hence, the limiting normalized solution $(U,P)$ satisfies the Stokes equations. The better regularity of $(U,P)$ is used to identify a good scale and yield the contradiction.

\emph{Step 2. Campanato estimate}, or, \emph{improve the decay of $Y_{\rm osc}(R)$}. The goal of this step is to demonstrate that, for all $0 < \alpha \ll 1$, smallness of $Y(1)$ and the coefficient $a$ implies
\begin{equation}
    \label{eq:thecampanatoest}
Y_{\rm osc}(R) \les_\alpha R^{\alpha} Y_{\rm osc}(1) \text{ for all } R \leq 1,
\end{equation}
since this estimate at every point is \emph{equivalent} to H{\"o}lder continuity.

Again, it is enough to show the improvement over a single scale, see Lemma~\ref{lem:interiorcampanato}. However, notice that smallness of $Y(1)$, rather than $Y_{\rm osc}(1)$, is required. The reason is the following. If the mean $(v)_Q$ is very large, then we expect that the drift carries information into the domain very quickly, which makes it more difficult to localize the solution. Therefore, to iterate the oscillation lemma, we require the scale-invariant Morrey estimate on $Y(R)$, that is,~\eqref{eq:themorreyest} with $\alpha = -1$, see the proof of  Lemma~\ref{lem:iteratedestimates}. This is why Lin's method contains two separate compactness arguments.

In the contradiction argument for the oscillation, one analyzes normalized solutions $v^{(k)}$ after subtracting off the mean. This introduces a significant new term $(v^{(k)})_Q \cdot \nabla a^{(k)} / \varepsilon_k$, which may not converge to zero but rather contributes a forcing term $\div F$ to the Stokes system for the limiting normalized solution $(U,P)$. When $a^{(k)} \in L^{5m/3}$ is subcritical ($m > 3$), the solution $(U,P)$ is H{\"o}lder continuous, and we can conclude. When $m=3$, this argument fails, so we stop at Step~1.\footnote{One can construct unbounded solutions to the heat equation in three dimensions with forcing $\div f$, where $f \in L^5$, and which belong to all $L^p$, $p < +\infty$. The Stokes equations are likely no better.}

The above two steps are interior estimates. In the boundary setting, there is a third step.

\emph{Step 3. Boundary Morrey estimate}, see Lemma~\ref{lem:boundarymorrey}. Due to the no-slip conditions, the quantity $Y^+(R)$ can decay as $R \to 0^+$, and one simply uses $Y^+(R)$ to control $Y^+_{\rm osc}(R)$.

\bigskip

Let $\alpha_0 = 2-2/\zeta_t$.
\begin{lemma}[Interior Morrey estimate]
\label{lem:interiormorrey}
Let $m\geq 3$. Let $(v,q)$ be a suitable weak solution of~\eqref{eq:perturbednse} on $Q$ in the sense of Definition~\ref{def.sws}.
For all $\alpha \in [-1,0)$, there exist constants $\varepsilon_0, \theta_0 \in (0,1)$ satisfying the following property.
If \begin{equation}
	Y(1) + \norm{a}_{\mathbf{A}_m(Q)} \leq \varepsilon_0,
\end{equation}
then
\begin{equation}
	Y(\theta_0) \leq \theta_0^\alpha Y(1).
\end{equation}
\end{lemma}
The above lemma is typically iterated with $\alpha = -1$, which corresponds to producing a critical bound at small scales.

\begin{lemma}[Boundary Morrey estimate]
\label{lem:boundarymorrey}
Let $m\geq 3$. Let $(v,q)$ be a suitable weak solution of~\eqref{eq:perturbednse} on $Q^+$ in the sense of Definition~\ref{def.sws}.
For all $\alpha \in (0,\alpha_0)$, there exist constants $\varepsilon_0^+, \theta_0^+ \in (0,1)$ satisfying the following property. If
\begin{equation}
	Y^+(1) + \norm{a}_{\mathbf{A}_m(Q^+)} \leq \varepsilon_0^+,
\end{equation}
then
\begin{equation}
	Y^+(\theta_0^+) \leq (\theta_0^+)^{\alpha} Y^+(1).
\end{equation}
\end{lemma}

Notice that in the above lemma $\alpha$ is positive, contrary to Lemma~\ref{lem:interiormorrey}. Let us stress once again that this fact relies on the no-slip boundary condition for $v$.

Let $\alpha_1 = \min(\alpha_0,1-3/m)$.
\begin{lemma}[Interior Campanato estimate]
\label{lem:interiorcampanato}
Let $m > 3$ and $(v,q)$ be a suitable weak solution of~\eqref{eq:perturbednse} on $Q$ in the sense of Definition~\ref{def.sws}.
For all $\alpha \in (0,\alpha_1)$, there exist constants $\varepsilon_1, \theta_1 \in (0,1)$ satisfying the following property.
If
\begin{equation}
	 \lvert (v)_{Q} \rvert + Y_\osc(1) + \norm{a}_{L^{5m/3}(Q)} \leq \varepsilon_1,
\end{equation} then
\begin{equation}
	Y_\osc(\theta_1) \leq \theta_1^{\alpha} \left( Y_\osc(1) +  \lvert (v)_{Q} \rvert \norm{a}_{L^{5m/3}(Q)} \right).
\end{equation}
\end{lemma}
Notice that the smallness of the mean velocity $|(v)_Q|$ is not propagated to small scales by Lemma~\ref{lem:interiorcampanato}. Rather, it is propagated by Lemma~\ref{lem:interiormorrey}.

We also require a translated, rescaled, and iterated version of the same lemmas:

\begin{lemma}[Iterated estimates]
	\label{lem:iteratedestimates}
	Let $(v,q)$ be a suitable weak solution of~\eqref{eq:perturbednse} on $Q_R^\iota(z_0)$ in the sense of Definition~\ref{def.sws}.
For all $\beta \in [-1,0)$ and $\alpha \in (0,\alpha_0)$, there exist a constant $\bar{\varepsilon}^\iota_0 > 0$ satisfying the following property.
If
\begin{equation}
	RY^\iota(z_0,R) + \norm{a}_{\mathbf{M}_m(Q_R^\iota(z_0))} \leq \bar{\varepsilon}^\iota_0,
\end{equation}
then, for all $r \in (0,R)$, we have
\begin{equation}
	Y^\iota(z_0,r) \les_\beta (r/R)^\beta Y^\iota(z_0,R).
\end{equation}
If also $\iota = {\rm bd}$, then
\begin{equation}
	Y^+(z_0,r) \les_\alpha (r/R)^\alpha Y^+(z_0,R).
\end{equation}
Finally, when $m > 3$ and $\alpha \in (0,\alpha_1)$, we have
\begin{equation}
	Y^\iota_\osc(z_0,r) \les_\alpha (r/R)^\alpha Y^\iota(z_0,R)
\end{equation}
in the interior and boundary settings.
\end{lemma}

\begin{proof}[Proof of Lemma~\ref{lem:interiormorrey} and Lemma~\ref{lem:boundarymorrey}]
We prove Lemma~\ref{lem:interiormorrey} and Lemma~\ref{lem:boundarymorrey} in tandem.
Let $\iota \in \{ {\rm int}, {\rm bd} \}$. If $\iota = {\rm int}$, let $\alpha \in [-1,0)$. If $\iota = {\rm bd}$, let $\alpha \in (0,\alpha_0)$.

\emph{1. Set-up}.
For contradiction, suppose that for each $\bar{\theta} \in (0,1/2)$, there exists a sequence $(v^{(k)},q^{(k)})$ of solutions to the perturbed Navier-Stokes equations in $Q^\iota$ with lower order terms $(a^{(k)})$ satisfying
\begin{equation}
	\varepsilon_k := Y^\iota(1,v^{(k)},q^{(k)}) \to 0^+,
\end{equation}
\begin{equation}
	\norm{a^{(k)}}_{\mathbf{A}_m(Q^\iota)} \to 0
\end{equation}
and
\begin{equation}
	Y^\iota(\bar{\theta},v^{(k)},q^{(k)}) > \bar{\theta}^\alpha \varepsilon_k.
\end{equation}
In Step 6 below, $\bar\theta$ will be fixed according to the limit problem (which is independent of $(v^{(k)},q^{(k)})$). 
 To capture the leading order terms in the PDE satisfied by $(v^{(k)},q^{(k)})$, we define
\begin{equation}
	w^{(k)} = \frac{v^{(k)}}{\varepsilon_k}, \quad \pi^{(k)} = \frac{q^{(k)}-
	[q^{(k)}]_{B^\iota}} {\varepsilon_k}.
\end{equation}
Then
\begin{equation}\label{e.contra1}
	Y^\iota(1,w^{(k)},\pi^{(k)}) = 1
\end{equation}
and
\begin{equation}
	 Y^\iota(\theta,w^{(k)},\pi^{(k)}) > \bar{\theta}^\alpha.
\end{equation}
Moreover, $(w^{(k)},\pi^{(k)})$ solves
\begin{equation}
	\label{eq:eqnforwk}
\left\lbrace
\begin{aligned}
	\p_t w^{(k)} - \Delta w^{(k)} + (\varepsilon_k w^{(k)} + a^{(k)}) \cdot \nabla w^{(k)} + w^{(k)} \cdot \nabla a^{(k)} + \nabla \pi^{(k)} &= 0 \\
	\div w^{(k)} &= 0.
	\end{aligned}
	\right.
\end{equation}
When $\iota = {\rm bd}$, we also require $w^{(k)} = 0$ on $\p_{\rm flat} Q^+$.

\emph{2. Energy estimates}. 
Let $1/2 \leq r < R \leq 1$ and $\phi \in C^\infty_0(Q_R)$ satisfying $\phi \equiv 1$ on $Q_r$, $0 \leq \phi \leq 1$ globally, and $|\p_t \phi| + |\nabla \phi|^2 + |\nabla^2 \phi| \les 1/(R-r)^2$. Let $\Phi = \phi^2$.

The local energy inequality becomes
\begin{equation}
\label{eq:localenergyineqwk}
\begin{aligned}
	&\int_{B^\iota} |w^{(k)}(x,t)|^2 \Phi(x,t) \, dx +2 \iint_{{B^\iota} \times (-1,t)} |\nabla w^{(k)}|^2 \Phi \, dx \, ds 
	\\
	&\leq \iint_{{B^\iota} \times (-1,t)} |w^{(k)}|^2 (\p_t + \Delta) \Phi  \, dx \, ds\\
	&\quad+ \iint_{{B^\iota} \times (-1,t)} (\varepsilon_k |w^{(k)}|^2 + 2 \pi^{(k)}) w^{(k)} \cdot \nabla \Phi \, dx \, ds \\
	&\quad+ \iint_{{B^\iota} \times (-1,t)} |w^{(k)}|^2 a^{(k)} \cdot \nabla \Phi \, dx \, ds\\
	&\quad+ 2 \iint_{{B^\iota} \times (-1,t)} a^{(k)} \otimes w^{(k)} : (\Phi \nabla  w^{(k)}  + w^{(k)} \otimes \nabla \Phi) \, dx \, ds.
	\end{aligned}
\end{equation}
We have
\begin{equation}
	\label{eq:movingphiinside}
\begin{aligned}
	\iint_{{B^\iota} \times (-1,t)} |\nabla (w^{(k)} \phi)|^2 \, dx \, ds &\les \iint_{{B^\iota} \times (-1,t)} |\nabla w^{(k)}|^2 \Phi + |w^{(k)}|^2 |\nabla \phi|^2 \, dx \, ds.
	\end{aligned}
\end{equation}
Together,~\eqref{eq:localenergyineqwk} and~\eqref{eq:movingphiinside} yield
\begin{equation}
\label{eq:localenergyineqwkrewritten}
\begin{aligned}
	&\sup_{t \in (-1,0)} \int_{B^\iota} |w^{(k)}(x,t) \phi (x,t)|^2 \, dx + \iint_{Q^\iota} |\nabla (w^{(k)} \phi)|^2 + |\nabla w^{(k)}|^2 \Phi   \, dx \, ds 
	\\
	& \les \iint_{Q^\iota} |w^{(k)}|^2 (|\p_t \Phi| + |\Delta \Phi| + |\nabla \phi|^2) \, dx \, ds\\
	&\quad+ \iint_{Q^\iota} (\varepsilon_k |w^{(k)}|^2+ |\pi^{(k)}|) |w^{(k)}| |\nabla \Phi| \, dx \, ds \\
	&\quad+\iint_{Q^\iota} |w^{(k)}|^2 |a^{(k)}| |\nabla \Phi| \, dx \, ds + \iint_{Q^\iota} |a^{(k)}| |w^{(k)}| |\nabla w^{(k)}| \Phi  \, dx \, ds.
	\end{aligned}
\end{equation}
In particular, Sobolev's embedding and elementary interpolation in Lebesgue spaces yield
\begin{equation}
	\label{eq:interpest}
	C_0 \norm{w^{(k)} \phi}_{L^q_t L^p_x(Q^\iota)}^2 \leq \text{right-hand side of } \eqref{eq:localenergyineqwkrewritten}
\end{equation}
as long as $q \in [2,\infty]$ and $\#(p,q) = -3/2$. For example, we may set $p=q=10/3$.

We now estimate each term on the right-hand side of~\eqref{eq:localenergyineqwkrewritten}. Immediately,
\begin{equation}
	\begin{aligned}
	&\iint_{Q^\iota} |w^{(k)}|^2 (|\p_t \Phi| + |\Delta \Phi| + |\nabla \phi|^2) \, dx \, ds + 2 \iint_{Q^\iota} |\pi^{(k)}|  |w^{(k)}| |\nabla \Phi| \, dx \, ds \\
	&\quad \les \frac{Y^\iota(1,w^{(k)},\pi^{(k)})^2}{(R-r)^2}.
	\end{aligned}
\end{equation}
Next, we estimate the terms involving $a^{(k)}$:
\begin{equation}
\begin{aligned}
	&\iint_{Q^\iota_R \setminus Q^\iota_r} |w^{(k)} \phi| |w^{(k)}|  |\nabla \phi|  |a^{(k)}|  \, dx \, ds \\
	&\quad \les \frac{1}{R-r} \norm{w^{(k)}}_{L^2(Q^\iota_R \setminus Q^\iota_r)} \norm{w^{(k)} \phi}_{L^{\frac{10}{3}}(Q^\iota)} \norm{a^{(k)}}_{L^{5m/3}(Q^\iota)} \\
	&\quad \les \frac{Y^\iota(1,w^{(k)},\pi^{(k)})^2}{(R-r)^2} + \norm{w^{(k)} \phi}_{L^{\frac{10}{3}}(Q^\iota)}^2 \norm{a^{(k)}}_{L^{5m/3}(Q^\iota)}^2
	\end{aligned}
\end{equation}
and
\begin{equation}
	\iint_{Q^\iota} |a^{(k)}| |w^{(k)}| |\nabla w^{(k)}| \phi^2 \, dx \, ds \les \norm{\phi \nabla w^{(k)}}_{L^2(Q^\iota)} \norm{w^{(k)} \phi}_{L^{\frac{10}{3}}(Q^\iota)} \norm{a^{(k)}}_{L^{5m/3}(Q^\iota)}.
\end{equation}
Finally, we estimate the term\footnote{Similar computations arise in the De Giorgi--Nash--Moser theory for parabolic equations with divergence-free drift $b$, specifically, in the treatment of the drift term in the proof of the Cacciopolli inequality~\cite{albritton2021regularity}.} involving $|w^{(k)}|^3$:
\begin{equation}
\begin{aligned}
	&\varepsilon_k \iint_{Q^\iota_R \setminus Q^\iota_r}  |w^{(k)}|^3 |\nabla \Phi| \, dx \, ds \\
	&\quad \les \frac{\varepsilon_k}{R-r} \norm{w^{(k)}}_{L^{\xi_t}_t L^{\xi_x}_x(Q^\iota_R \setminus Q^\iota_r)} \norm{w^{(k)}}^2_{L^{2 \zeta_t}_t L^{2 \zeta_x}_x(Q^\iota_R \setminus Q^\iota_r)} \\
	&\quad\les \frac{\varepsilon_k}{R-r} Y^\iota(1,w^{(k)},\pi^{(k)}) \norm{w^{(k)}}^{2\kappa}_{L^q_t L^p_x(Q^\iota_R \setminus Q^\iota_r)} \norm{w^{(k)}}^{2(1-\kappa)}_{L^2_{t,x}(Q^\iota_R \setminus Q^\iota_r)} \\
	&\quad\les \frac{O(1)}{(R-r)^{\frac{1}{1-\kappa}}} + \varepsilon_k^{\frac{1}{\kappa}} \norm{w^{(k)}}_{L^q_t L^p_x(Q^\iota_R \setminus Q^\iota_r)}^2.
	\end{aligned}
\end{equation}
where $\kappa \in (0,1)$, $\#(p,q) = -3/2$, and $q \in (2,+\infty]$ is now fixed.\footnote{\label{foot.num}The numerology is as follows: $\kappa \in (0,1)$ is chosen to satisfy
\begin{equation}
    \label{eq:zetaxzetat}
	\#(2\zeta_x,2\zeta_t) = \kappa (-3/2) + (1-\kappa) (-5/2),
\end{equation}
which is possible because $\#(2\zeta_x,2\zeta_t) = -7/4+\delta_0/2$. Here, $-3/2$ is the scaling number associated with the energy space and $-5/2 = \#(2,2)$. Then $(p,q)$ are defined by
\begin{equation}
    \label{eq:moreaboutzetatzetax}
	\frac{1}{2\zeta_t} = \frac{\kappa}{q} + \frac{1-\kappa}{2}, \quad
	\frac{1}{2\zeta_x} = \frac{\kappa}{p} + \frac{1-\kappa}{2}.
\end{equation}
It follows from~\eqref{eq:zetaxzetat} and~\eqref{eq:moreaboutzetatzetax} that $\#(p,q) = -3/2$ and $q \in (2,+\infty]$.}

Since
\begin{equation}
    \norm{a^{(k)}}_{L^{5m/3}(Q^\iota)} \to 0\quad\mbox{as}\quad k \to +\infty,
\end{equation} various terms containing it may be absorbed into the left-hand side of~\eqref{eq:localenergyineqwkrewritten}. This yields
\begin{equation}
\label{eq:localenergyineqwkrewrittenagain}
\begin{aligned}
	&\sup_{t \in (-r^2,0)} \int_{B^\iota_r} |w^{(k)}(x,t)|^2 \, dx + \iint_{Q^\iota_r} |\nabla w^{(k)}|^2  \, dx \, ds + \norm{w^{(k)}}_{L^q_t L^p_x(Q^\iota_r)}^2 \\
	&\quad\les \frac{O(1)}{(R-r)^{\frac{1}{1-\kappa}}} + \varepsilon_k^{\frac{1}{\kappa}} \norm{w^{(k)}}_{L^q_t L^p_x(Q^\iota_R \setminus Q^\iota_r)}^2.
	\end{aligned}
\end{equation}
When $k \geq k_0 \gg 1$, we have $\varepsilon_k \ll 1$, so we may iterate the inequality~\eqref{eq:localenergyineqwkrewrittenagain} `outward' along a well-chosen increasing sequence of scales $r_j$, $j = 0,1,2,\hdots$, with $r_0 = 3/4$ and $r_j \to 1$ as $j \to +\infty$, see~\cite[p. 191, Lemma 6.1]{giustibook}. Hence, when $k \gg 1$,
\begin{equation}
\label{eq:localenergyineqwkrewrittenagainagain}
\begin{aligned}
	&\sup_{t \in (-9/16,0)} \int_{B^\iota_{3/4}} |w^{(k)}(x,t)|^2 \, dx + \iint_{Q^\iota_{3/4}} |\nabla w^{(k)}|^2  \, dx \, ds \les 1. 
	\end{aligned}
\end{equation}
Therefore, along a subsequence (without relabeling):
\begin{equation}
	w^{(k)} \wstar U \text{ in } L^\infty_t L^2_x \cap L^2_t H^1_x(Q^\iota_{3/4})
\end{equation}
\begin{equation}
	\pi^{(k)} \wstar P \text{ in } L^{\zeta_t}_t L^{\zeta_x}_x(Q^\iota)
\end{equation}
and
\begin{equation}
    \label{eq:gloriousstokesestimate}
	\sup_{t \in (-9/16,0)} \int_{B^\iota_{3/4}} |U(x,t)|^2 \, dx + \iint_{Q^\iota_{3/4}} |\nabla U|^2  \, dx \, ds + Y^\iota(1,U,P) \leq 1.
\end{equation}
If $\iota = {\rm bd}$, we also have $U|_{\p_{\rm flat} Q^+_{3/4}} = 0$. Moreover, after analyzing the convergence of each term in~\eqref{eq:eqnforwk}, we find
\begin{equation}
	\label{eq:UsatisfiesStokes}
\left\lbrace
\begin{aligned}
	\p_t U - \Delta U + \nabla P &= 0 \quad \text{ in } Q^\iota_{3/4} \\
	\div U &= 0 \quad \text{ in } Q^\iota_{3/4}.
	\end{aligned}
	\right.
\end{equation}

\emph{3. Time derivative estimates}. Using the above estimates and
\begin{equation}
	\p_t w^{(k)} = \Delta w^{(k)} - \div [ (\varepsilon_k w^{(k)} + a^{(k)}) \otimes w^{(k)} + w^{(k)} \otimes a^{(k)} ] - \nabla \pi^{(k)},
\end{equation}
we estimate the time derivative:
\begin{equation}
	\norm{\p_t w^{(k)}}_{L^{1+\nu}_t W^{-1,1+\nu}_x(Q^\iota_{3/4})} \les 1,
\end{equation}
where $0 < \nu \ll 1$. 
The Aubin--Lions lemma yields, up to a subsequence,\footnote{The requisite chain of embeddings is $H^1(B_{1/2}^\iota) \overset{\rm cpt}{\into} L^2(B_{1/2}^\iota) \into W^{-1,1+\nu}(B_{1/2}^\iota)$.}
\begin{equation}
	w^{(k)} \to U \text{ in } L^2_{t,x}(Q_{3/4}^\iota).
\end{equation}
By interpolation with the energy norm, we also have
\begin{equation}
	\label{eq:strongconvofwk}
	w^{(k)} \to U \text{ in } L^q_t L^p_x(Q_{3/4})
\end{equation}
when $\#(p,q) < -3/2$ and $q \in [2,+\infty)$.

\emph{4. Estimates for the limit equation}. By~\eqref{eq:gloriousstokesestimate} and local maximal regularity \cite[Theorem 1.2]{Ser09} for the time-dependent Stokes equations~\eqref{eq:UsatisfiesStokes} in $Q^\iota_{3/4}$,
\begin{equation}
	\norm{\p_t U, \nabla^2 U, \nabla U, U, \nabla P}_{L^{\zeta_t}_t L^p_x(Q^\iota_{1/2})} \les_p 1
\end{equation}
for all $p \in [1,+\infty)$, where Ehrling's inequality~\cite[p. 77]{galdi} is used to estimate $\nabla U$ in terms of $U$ and $\nabla^2 U$. In particular, parabolic Sobolev embedding into H{\"o}lder spaces (see~\cite[Lemma B.1]{albrittonbarkerlocalregII}) yields
\begin{equation}\label{e.holderU}
	\norm{U}_{C^{\alpha'}_\para(Q^\iota_{1/2})} \les_{\alpha'} 1
\end{equation}
for all $\alpha' \in (0,\alpha_0)$, where we recall that $\alpha_0 = 2 - 2/\zeta_t$. 
We claim that there exists $\theta_1 \in (0,1/4)$ such that
\begin{equation}
	\label{eq:velocityclaim}
	\theta^{\#\xi} \norm{U}_{L^{\xi_t}_t L^{\xi_x}_x(Q^\iota_\theta)} \leq \theta^\alpha/8
\end{equation}
for all $\theta \in (0,\theta_1]$. Here, we will distinguish between the interior and boundary cases. When $\iota = {\rm int}$, we took $\alpha \in [-1,0)$, and~\eqref{eq:velocityclaim} follows from $\norm{U}_{L^\infty(Q_{1/2})} \les 1$.
 When $\iota = {\rm bd}$, we took $\alpha \in (0,\alpha_0)$, and~\eqref{eq:velocityclaim} follows from $\norm{U}_{C^{\alpha'}_\para(Q^+_{1/2})} \les_{\alpha'} 1$ for all $\alpha' \in (\alpha,\alpha_0)$ and the no-slip boundary condition.

Finally, by the strong convergence~\eqref{eq:strongconvofwk} of $w^{(k)} \to U$ in $L^{\xi_t}_t L^{\xi_x}_x(Q^\iota_{3/4})$,
\begin{equation}
	\label{eq:velocitystrongconvergence}
	\limsup_{k \to +\infty} \theta^{\#\xi} \norm{w^{(k)}}_{L^{\xi_t}_t L^{\xi_x}_x(Q^\iota_\theta)} \leq \theta^\alpha/4.
\end{equation}

\emph{5. Pressure estimates}. Next, we decompose the pressure. Let
\begin{equation}
\begin{aligned}
	f^{(k)} &= \varepsilon_k w^{(k)} \cdot \nabla w^{(k)} + a^{(k)} \cdot \nabla w^{(k)} + w^{(k)} \cdot \nabla a^{(k)}.
	\end{aligned}
\end{equation}
We will show
\begin{equation}
	\norm{f^{(k)}}_{L^{\zeta_t}_t L^{\ell_x}_x(Q_{3/4}^\iota)} \to 0 \text{ as } k \to +\infty
\end{equation}
where $\ell_x$ is defined by $\#(\ell_x,\zeta_t) = -4$. We estimate term-by-term:
\begin{equation}
	\label{eq:festimate1}
	\varepsilon_k \norm{w^{(k)} \cdot \nabla w^{(k)}}_{L^{\zeta_t}_t L^{\ell_x}_x(Q_{3/4}^\iota)}  \les \varepsilon_k \norm{w^{(k)}}_{L^{q_1}_t L^{p_1}_x(Q_{3/4}^\iota)} \norm{\nabla w^{(k)}}_{L^2_{t,x}(Q_{3/4}^\iota)} \les \varepsilon_k \to 0,
\end{equation}
\begin{equation}
	\norm{a^{(k)} \cdot \nabla w^{(k)}}_{L^{\zeta_t}_t L^{\ell_x}_x(Q_{3/4}^\iota)} \les \norm{a^{(k)}}_{L^{q_1}_t L^{p_1}_x(Q^\iota_{3/4})} \norm{\nabla w^{(k)}}_{L^2(Q^\iota_{3/4})} \les \norm{a^{(k)}}_{\mathbf{A}_m(Q^\iota)} \to 0,
\end{equation}
and
\begin{equation}
	\norm{w^{(k)} \cdot \nabla a^{(k)}}_{L^{\zeta_t}_t L^{\ell_x}_x(Q^\iota_{3/4})} \les \norm{w^{(k)}}_{L^q_t L^p_x} \norm{\nabla a^{(k)}}_{L^{q_2}_t L^{p_2}_x(Q_{3/4}^\iota)} \les \norm{a^{(k)}}_{\mathbf{A}_m(Q^\iota)} \to 0,
\end{equation}
where $\#(p,q) = -3/2$, $1/\zeta_t = 1/q+1/q_2$, and $q \in [2,+\infty]$ is fixed. 
Recall that $(p_1,q_1)$, $(p_2,q_2)$ are defined above~\eqref{e.norma}.

Consider the solution $(\tilde{w}^{(k)},\tilde{\pi}^{(k)})$ of the Stokes equations
\begin{equation}
\left\lbrace
\begin{aligned}
	\p_t \tilde{w}^{(k)} - \Delta \tilde{w}^{(k)} + \nabla \tilde{\pi}^{(k)} &= - \chi f^{(k)} && \text{ in } \R^3_{\iota} \times (-1,0) \\
		\div \tilde{w}^{(k)} &= 0 && \text{ in } \R^3_{\iota} \times (-1,0)  \\
		\tilde{w}^{(k)}(\cdot,-1) &= 0 &&
		\end{aligned}
		\right.
\end{equation}
where $\R^3_{\rm int} = \R^3$, $\R^3_{\rm bd} = \R^3_+$, and $\chi \in C^\infty_0(\R^{3+1})$ with $\chi \equiv 1$ on $Q_{5/8}$ and $\chi \equiv 0$ outside of $Q_{3/4}$. When $\iota = {\rm bd}$, we impose also the no-slip condition $\tilde{w}^{(k)}|_{\p \R^3_+}(\cdot,t) = 0$ for all $t \in (-1,0)$. Global maximal regularity implies
\begin{equation}\label{e.globmaxreg}
	\norm{\nabla \tilde{\pi}^{(k)}, \p_t \tilde{w}^{(k)}, \nabla^2 \tilde{w}^{(k)}, \nabla \tilde{w}^{(k)}, \tilde{w}^{(k)}}_{L^{\zeta_t}_t L^{\ell_x}_x(\R^3_\iota \times (-1,0))} \to 0 \text{ as } k \to +\infty
\end{equation}
and, by a Poincar{\'e}-Sobolev inequality, for all $\theta \in (0,5/8]$,
\begin{equation}
	\label{eq:strongconvofpressure}
	\norm{\tilde{\pi}^{(k)} - [\tilde{\pi}^{(k)}]_{B_\theta^\iota}}_{L^{\zeta_t}_t L^{\zeta_x}_x(Q^\iota_\theta)} \to 0 \text{ as } k \to +\infty,
\end{equation}
since $\# \zeta = -7/2 + \delta_0  < -3 = 1+\#(\ell_x,\zeta_t)$.
Let
\begin{equation}
	\label{eq:harmdef}
	w^{(k)}_{\harm} = w^{(k)} - \tilde{w}^{(k)}, \quad \pi^{(k)}_\harm = \pi^{(k)} - [\pi^{(k)}]_{B^\iota} - \tilde{\pi}^{(k)}.
\end{equation}
Then
\begin{equation}
    \label{eq:ilovethestokesequations}
\left\lbrace
	\begin{aligned}
	\p_t w^{(k)}_\harm - \Delta w^{(k)}_\harm + \nabla \pi^{(k)}_\harm &= 0 && \text{ in } Q^\iota_{5/8} \\
	\div w^{(k)}_\harm &= 0 && \text{ in } Q^\iota_{5/8}
	\end{aligned}
	\right.
\end{equation}
with $w^{(k)}_\harm|_{\p_{\rm flat} Q^+_{5/8}} = 0$ in the boundary case. 
From~\eqref{eq:harmdef} and the estimates \eqref{e.contra1}, \eqref{e.globmaxreg} and \eqref{eq:strongconvofpressure}, we know
\begin{equation}
    \label{eq:theestimatesweknow}
	\norm{w^{(k)}_\harm, \nabla w^{(k)}_\harm}_{L^{1+\nu}_{t,x}(Q^\iota_{3/4})} + \norm{\pi^{(k)}_\harm - [\pi^{(k)}_\harm]_{B^\iota_{3/4}}}_{L^{\zeta_t}_t L^{1+\nu}_x(Q^\iota_{3/4})} \les 1,
	\end{equation}
	where $0 < \nu \ll 1$.
Local maximal regularity estimates for solutions of~\eqref{eq:ilovethestokesequations} satisfying~\eqref{eq:theestimatesweknow} imply
\begin{equation}
	\norm{\nabla \pi^{(k)}_\harm}_{L^{\zeta_t}_t L^p_x(Q^\iota_{1/2})} \les_p 1 \text{ for all } p \in [1,+\infty).
\end{equation}
Now, by Poincar{\'e}'s inequality,
\begin{equation}
	\label{eq:whatweknowaboutpiharm}
	\begin{aligned}
	\norm{\pi^{(k)}_\harm - [\pi^{(k)}_\harm]_{B^\iota_\theta}}_{L^{\zeta_t}_t L^{\zeta_x}_x(B^\iota_\theta \times (-1/4,0))} &\les \theta \norm{\nabla \pi^{(k)}_\harm}_{L^{\zeta_t}_t L^{\zeta_x}_x(B^\iota_\theta \times (-1/4,0))} \\
	&\les_p \theta^{1+\frac{3}{\zeta_x}-\frac{3}{p}}\\
	&\les_p \theta^{|\#\zeta|-1} \theta^{\alpha_0 - \frac{3}{p}}.
	\end{aligned}
\end{equation}
 In particular, whenever $\alpha < \alpha_0$, there exists $\theta_2 = \theta_2(\alpha)\in (0,1/4)$ such that, for all $\theta \in (0,\theta_2]$ and $k \gg 1$, we have
\begin{equation}
    \label{eq:piharmbound}
	\theta^{\#\zeta+1} \norm{\pi^{(k)}_\harm - [\pi^{(k)}_\harm]_{B^\iota_\theta}}_{L^{\zeta_t}_t L^{\zeta_x}_x(Q^\iota_\theta)} \leq \frac{\theta^\alpha}{4}.
\end{equation}
Hence, using the strong convergence of $\tilde{\pi}^{(k)}$ in~\eqref{eq:strongconvofpressure}, we have
\begin{equation}
	\label{eq:pressurelimsupassertion}
	\limsup_{k \to +\infty} \theta^{\#\zeta+1} \norm{\pi^{(k)} - [\pi^{(k)}]_{B^\iota_\theta}}_{L^{\zeta_t}_t L^{\zeta_x}_x(Q^\iota_\theta)} \leq \frac{\theta^\alpha}{4}.
\end{equation}

\emph{6. Contradiction}. Let $\theta_0^\iota = \min(\theta_1,\theta_2)$. Then, in light of~\eqref{eq:velocitystrongconvergence} and~\eqref{eq:pressurelimsupassertion}, we have
\begin{equation}
	\limsup_{k \to +\infty} Y^\iota(\theta_0^\iota,w^{(k)},\pi^{(k)}) \leq (\theta_0^\iota)^\alpha/2.
\end{equation}
Choosing $\bar{\theta} = \theta_0^\iota$ at the beginning of the proof yields the desired contradiction.
\end{proof}

\begin{proof}[Proof of Lemma~\ref{lem:interiorcampanato}]
As in Lemma~\ref{lem:interiormorrey}, we suppose otherwise. That is, for each $\alpha \in (0,\alpha_1)$ and $\bar{\theta} \in (0,1/2)$, there exists a sequence $(v^{(k)},q^{(k)})$ of solutions to the perturbed Navier-Stokes equations in $Q$ with lower order terms $(a^{(k)})$ satisfying
\begin{equation}
	\varepsilon_k := Y_\osc(1,v^{(k)},q^{(k)}) + \lvert (v^{(k)})_Q \rvert \norm{a}_{\mathbf{A}_m(Q)} \to 0^+,
\end{equation}
\begin{equation}
    \label{eq:vkmeantozero}
	|(v^{(k)})_{Q}| \to 0,
\end{equation}
\begin{equation}
	\norm{a^{(k)}}_{\mathbf{A}_m(Q)} \to 0,
\end{equation}
and
\begin{equation}
	Y_\osc(\theta,v^{(k)},q^{(k)}) > \varepsilon_k  \theta^\alpha.
\end{equation}
We define $(w^{(k)},\pi^{(k)})$ by subtracting off the mean velocity as follows:
\begin{equation}
	w^{(k)} = \frac{v^{(k)}-(v^{(k)})_Q}{\varepsilon_k}, \quad \pi^{(k)} = \frac{q^{(k)}  - [q^{(k)}]_B}{\varepsilon_k}.
\end{equation}
Then
\begin{equation}
	Y(1,w^{(k)},\pi^{(k)}) = Y_\osc(1,w^{(k)},\pi^{(k)}) \leq 1
\end{equation}
and
\begin{equation}
	Y_\osc(\theta,w^{(k)},\pi^{(k)}) > \theta^\alpha.
\end{equation}
The remainder of the proof is similar to the proof of Lemma~\ref{lem:interiormorrey}, so we will sketch some arguments. The main difference is that
 $(w^{(k)},\pi^{(k)})$ solves the perturbed Navier-Stokes equations with two additional terms, namely, $(v^{(k)})_Q \cdot \nabla w^{(k)}$ and $(v^{(k)})_Q \cdot \nabla a^{(k)} / \varepsilon_k$:
\begin{equation}
\begin{aligned}
	&\p_t w^{(k)} - \Delta w^{(k)} + \left[ \varepsilon_k w^{(k)} + a^{(k)} + (v^{(k)})_Q \right] \cdot \nabla w^{(k)} \\
	&\quad + \left[w^{(k)} + \frac{(v^{(k)})_Q}{\varepsilon_k}\right ] \cdot \nabla a^{(k)} + \nabla \pi^{(k)} = 0, \quad \div w^{(k)} = 0 \, .
	\end{aligned}
\end{equation}
The term $(v^{(k)})_Q \cdot \nabla a^{(k)} / \varepsilon_k$ requires special attention, since \emph{a priori} it may not be converging to zero. Indeed, upon passing to a subsequence, we have
\begin{equation}
	F_k := a^{(k)} \otimes \frac{(v^{(k)})_Q}{\varepsilon_k} \wto F \text{ in } L^{5m/3}(Q),
\end{equation}
where $\norm{F}_{L^{5m/3}(Q)} \les 1$. This contributes a term $\div F$ to the limiting PDE.
 
The new terms above contribute the following to the right-hand side of the local energy inequality~\eqref{eq:localenergyineqwk}:
\begin{equation}
	\begin{aligned}
	&\iint_{B \times (-1,t)} |w^{(k)}|^2 (v^{(k)})_Q \cdot \nabla \Phi \, dx \, ds \\
	&\quad+ 2 \iint_{B \times (-1,t)} a^{(k)} \otimes \frac{(v^{(k)})_Q}{\varepsilon_k} : [\Phi \nabla w^{(k)} + w^{(k)} \otimes \nabla \Phi] \, dx \, ds.
	\end{aligned}
\end{equation}
We estimate this quantity in the following way:
\begin{equation}
\begin{aligned}
	\iint_Q |w^{(k)}|^2 |(v^{(k)})_Q| |\nabla \Phi| \, dx \, ds &\les \frac{Y(1,w^{(k)},\pi^{(k)})^2}{R-r}  |(v^{(k)})_Q| \\
	&\overset{\eqref{eq:vkmeantozero}}{\les} \frac{o(1)}{R-r} \text{ as } k \to +\infty,
	\end{aligned}
\end{equation}
and
\begin{equation}
\begin{aligned}
	& \frac{1}{\varepsilon_k} \iint_{Q} |a^{(k)}| |(v^{(k)})_Q| [\Phi |\nabla w^{(k)}| + |w^{(k)}| |\nabla \Phi|] \, dx \, ds  \\
	&\quad \leq \underbrace{\frac{C}{\varepsilon_k} \norm{a^{(k)}}_{L^{5m/3}(Q)} |(v^{(k)})_Q|}_{O(1)} \left[ \norm{\phi \nabla w^{(k)}}_{L^2(Q)} + \frac{Y(1,w^{(k)},\pi^{(k)})}{R-r} \right] \\
	&\quad \leq C\gamma^{-1} + \gamma \norm{\phi \nabla w^{(k)}}_{L^2(Q)}^2 + \frac{C}{R-r} \, ,
	\end{aligned}
\end{equation}
 where $\gamma > 0$ is a free parameter and we employ that $\Phi = \phi^2 \leq \phi$. Choosing $0 < \gamma \ll 1$, we may absorb $\gamma \norm{\phi \nabla w^{(k)}}_{L^2(Q)}^2$ into the left-hand side of~\eqref{eq:localenergyineqwkrewritten} and obtain the energy estimates as before.

The arguments concerning the time derivative apply nearly identically, so we next analyze the velocity regularity. As mentioned above, the term $\div F_k = (v^{(k)})_Q \cdot \nabla a^{(k)} / \varepsilon_k$ contributes to the limiting Stokes system satisfied by $(U,P)$:
\begin{equation}
	\label{eq:limitingstokessystem}
\left\lbrace
\begin{aligned}
	\p_t U - \Delta U + \nabla P &= - \div F \\
		\div U &= 0 \, .
		\end{aligned}
		\right.
\end{equation}
Since $\norm{F}_{L^{5m/3}(Q)} \les 1$ and $m > 3$, we may use a standard bootstrapping procedure (see Lemma 2.2 in~\cite{jiasverakselfsim}, for example) to prove
\begin{equation}
	\norm{U}_{C^{\alpha'}_\para(Q_{1/2})} \les_{\alpha'} 1
\end{equation}
for all $\alpha' \in (0,\alpha_1)$, where $\alpha_1 = \min(\alpha_0,1-3/m)$.
By Campanato's characterization of H{\"o}lder spaces, for all $\alpha < \alpha_1$, there exists $\theta_2 = \theta_2(\alpha) \in (0,1/4]$ such that, whenever $\theta \in (0,\theta_2]$,
\begin{equation}
	\label{eq:Uoscest}
	\theta^{\#\zeta} \norm{U - (U)_{Q_\theta}}_{L^{\xi_t}_t L^{\xi_x}_x(Q_\theta)} \leq \frac{\theta^\alpha}{4} \, .
\end{equation}

We now discuss the pressure estimates. We incorporate $(v^{(k)})_Q \cdot \nabla w^{(k)}$ into $f^{(k)}$: \begin{equation}
\begin{aligned}
	f^{(k)} &= (\varepsilon_k w^{(k)} + (v^{(k)})_Q) \cdot \nabla w^{(k)} + a^{(k)} \cdot \nabla w^{(k)} + w^{(k)} \cdot \nabla a^{(k)} \, .
	\end{aligned}
\end{equation}
This new term can be estimated, for example, in the same way as $\varepsilon_k w^{(k)} \cdot \nabla w^{(k)}$. Again, $f^{(k)}$ converges strongly to zero in $L^{\zeta_t}_t L^{\ell_x}_x(Q_{3/4}^\iota)$ as $k \to +\infty$. We define and estimate $(\tilde{w}^{(k)},\tilde{\pi}^{(k)})$ as before, but to accommodate the $\div F_k$ term, we additionally consider $(\hat{w}^{(k)},\hat{\pi}^{(k)})$ solving
\begin{equation}
\left\lbrace
\begin{aligned}
	\p_t \hat{w}^{(k)} - \Delta \hat{w}^{(k)} + \nabla \hat{\pi}^{(k)} &= - \div (\chi F_k) && \text{ in } \R^3 \times (-1,0) \\
		\div \hat{w}^{(k)} &= 0 && \text{ in } \R^3 \times (-1,0)   \\
		\hat{w}^{(k)}(\cdot,-1) &= 0 \, . &&
		\end{aligned}
		\right.
\end{equation}
We have the representation formula
\begin{equation}
    \hat{\pi}^{(k)} = (-\Delta)^{-1} \div \div ( \chi F_k)
\end{equation}
and $\| \hat{\pi}^{(k)} \|_{L^{5m/3}(\R^3 \times (-1,0))} \les \| F_k \|_{L^{5m/3}(\R^3 \times (-1,0))} \les 1$ by Calder{\'o}n-Zygmund estimates. Then
\begin{equation}
    \| \hat{\pi}^{(k)} - [\hat{\pi}^{(k)}]_{B_\theta} \|_{L^{\zeta_t}_t L^{\zeta_x}_x(Q_\theta)} \les \theta^{|\#\zeta| - \frac{3}{m}} \| \hat{\pi}^{(k)} \|_{L^{5m/3}(Q_\theta)} \, ,
\end{equation}
and, in particular,
\begin{equation}
    \label{eq:pihatpressureest}
    \theta^{\#\zeta + 1}  \| \hat{\pi}^{(k)} - [\hat{\pi}^{(k)}]_{B_\theta} \|_{L^{\zeta_t}_t L^{\zeta_x}_x(Q_\theta)} \les \theta^{1-\frac{3}{m}} \, . 
\end{equation}
Finally, consider
\begin{equation}
    \pi^{(k)}_\harm = \pi^{(k)} - [\pi^{(k)}]_{B^\iota} - \tilde{\pi}^{(k)} - \hat{\pi}^{(k)} \, .
\end{equation}
Since $L^{5m/3}(Q) \subset L^{\zeta_t}_t L^{5m/3}_x(Q)$, the regularity of $\pi^{(k)}_\harm$ is dealt with as before to obtain~\eqref{eq:piharmbound}. By combining~\eqref{eq:pihatpressureest} and~\eqref{eq:piharmbound} with the strong convergence of $\tilde{\pi}^{(k)}$ to zero, we conclude that whenever $\alpha < \alpha_1$, there exists $\theta_3 = \theta_3(\alpha) \in (0,1/4)$ such that, for all $\theta \in (0,\theta_3]$, we have
\begin{equation}
    \label{eq:limitingpressureestagain}
	\limsup_{k \to +\infty} \theta^{\#\zeta+1} \norm{\pi^{(k)} - [\pi^{(k)}]_{B_\theta}}_{L^{\zeta_t}_t L^{\zeta_x}_x(Q_\theta)} \leq \frac{\theta^\alpha}{4}.
\end{equation}
Let $\theta_1 = \min(\theta_2,\theta_3)$. 
Combining~\eqref{eq:limitingpressureestagain} with the strong convergence $w^{(k)} \to U$ in $L^{\xi_t}_t L^{\xi_x}_x(Q_{3/4})$ and the H{\"o}lder estimate~\eqref{eq:Uoscest} on $U$, we have
\begin{equation}
	\limsup_{k \to +\infty} Y_\osc(\theta_1,w^{(k)},\pi^{(k)}) \leq \frac{\theta_1^\alpha}{2},
\end{equation}
which, upon setting $\bar{\theta} = \theta_1$, yields the desired contradiction.
\end{proof}

\begin{proof}[Proof of Lemma~\ref{lem:iteratedestimates}]
Without loss of generality, $R = 1$ and $z_0 = 0$. Let $\varepsilon > 0$ and
\begin{equation}
	Y^\iota(1) + \norm{a}_{\mathbf{M}_m(Q^\iota)} \leq \varepsilon.
\end{equation}
Recall that $\norm{a}_{\mathbf{M}_m(Q^\iota_r)} \leq \norm{a}_{\mathbf{M}_m(Q^\iota)}$ when $r \in (0,1)$.

\textit{1. Boundary case}.
Let $\alpha \in (0,\alpha_0)$. If $\varepsilon \leq \varepsilon_0^+$, then Lemma~\ref{lem:boundarymorrey} implies
\begin{equation}
	Y^+(\theta_0^+) \leq (\theta_0^+)^\alpha Y^+(1),
\end{equation}
where $\varepsilon_0^+,\theta_0^+ \in (0,1)$ are the constants in Lemma~\ref{lem:boundarymorrey}.
In particular,
\begin{equation}
\begin{aligned}
	\theta_0^+ Y^+(\theta_0^+) + \norm{a}_{\mathbf{M}_m(Q^+_{\theta_0^+})} &\leq (\theta_0^+)^{\alpha+1} Y^+(1) + \norm{a}_{\mathbf{M}_m(Q^+)} \\
	&\leq Y^+(1) + \norm{a}_{\mathbf{M}_m(Q^+)} \\
	&\leq \varepsilon_0^+ \, ,
	\end{aligned}
\end{equation}
and therefore we may apply a rescaled version of Lemma~\ref{lem:boundarymorrey} at scale $\theta_0^+$. Iterating in this fashion, we have
\begin{equation}
	Y^+((\theta_0^+)^k) \leq (\theta_0^+)^{k\alpha} Y^+(1)
\end{equation}
for all $k \in \N$. Filling in the intermediary scales, we obtain
\begin{equation}
	Y^+(r) \les_\alpha r^\alpha Y^+(1)
\end{equation}
for all $r \in (0,1)$. Since $Y^+_\osc \leq 2Y^+$, we also have
\begin{equation}
	Y^+_\osc(r) \les_\alpha r^\alpha Y^+(1),
\end{equation}
as desired.

\textit{2. Interior case}.
Let $\beta \in [-1,0)$. First, we prove a Morrey estimate as in the boundary case. If $\varepsilon \leq \varepsilon_0$, then Lemma~\ref{lem:interiormorrey} (with $\alpha = \beta$ in the statement) and reasoning as above yields
\begin{equation}
	\label{eq:wefilledintheintermediaryscales}
	Y(r) \les_\beta r^{\beta} Y(1).
\end{equation}

Next, we estimate the oscillation, which is special to $m > 3$. Let $\alpha \in (0,\alpha_1)$. Recall that $\alpha_1 = \min(1-3/m,\alpha_0)$. In~\eqref{eq:wefilledintheintermediaryscales}, we choose
\begin{equation}
    \label{eq:choiceofbetagamma}
    \gamma = (\alpha+\alpha_1)/2 \in (\alpha,\alpha_1), \quad \beta = \gamma-\alpha_1 \in [-1,0).
\end{equation}
We choose $0 < \varepsilon \ll_\alpha 1$ such that~\eqref{eq:wefilledintheintermediaryscales} implies
\begin{equation}
	\label{eq:smallnessoverscales}
	r |(v)_{Q_r}| + r Y_\osc(r) \leq \varepsilon_1/3
\end{equation}
for all $r \in (0,1]$, where $\varepsilon_1$ is as in Lemma~\ref{lem:interiorcampanato} (with $\alpha = \gamma$ in the statement). We also choose $\varepsilon \leq \varepsilon_1/3$. This means that the hypotheses of Lemma~\ref{lem:interiorcampanato} are satisfied at each scale $r \in (0,1]$, and we may freely apply it.
We will show
\begin{equation}
	\label{eq:Yosciteration}
	Y_\osc(\theta_1^k) \leq 3 \theta_1^{k\alpha} Y(1)
\end{equation}
for all $k \geq 0$. With this in hand, one may fill in the remaining scales:
\begin{equation}
	Y_\osc(r) \les_\alpha r^\alpha Y(1),
\end{equation}
for all $r \in (0,1)$, as desired.

The base case $Y_{\rm osc}(1) \leq 3 Y(1)$ of~\eqref{eq:Yosciteration} is obvious. Assume~\eqref{eq:Yosciteration} is valid for $k = k_0 \geq 0$. Then a rescaled version of Lemma~\ref{lem:interiorcampanato} (with $\alpha = \gamma$ in the statement) and the inductive hypothesis yield
\begin{equation}
\begin{aligned}
	Y_\osc(\theta_1^{k_0+1}) &\leq \theta_1^{\gamma} \left( Y_\osc(\theta_1^{k_0}) + |(v)_{Q_{\theta_1^{k_0}}}| \times \theta_1^{k_0(1-3/m)} \norm{a}_{L^{5m/3}(Q_{\theta_1^{k_0}})}  \right)\\
	 &\leq 3 \theta_1^{\gamma} \theta_1^{k_0 \alpha} Y(1) + C_1 \theta_1^{\gamma} \theta_1^{k_0 (\gamma-\alpha_1) } Y(1) \times \theta_1^{k_0 \alpha_1} \norm{a}_{L^{5m/3}(Q)} \\
	 &\leq 3 \theta_1^{(k_0+1)\alpha} Y(1) \times (\theta_1^{\gamma-\alpha} + C_1 \varepsilon/3) \\
	 &\leq 3 \theta_1^{(k_0+1)\alpha} Y(1)
	\end{aligned}
\end{equation}
when $0 < \varepsilon \ll_\alpha 1$. In the second inequality we used~\eqref{eq:wefilledintheintermediaryscales} with the choice of $(\gamma,\beta)$ in~\eqref{eq:choiceofbetagamma} and also $\alpha_1 \leq 1-3/m$.
This completes the induction and the proof.
\end{proof}

\begin{proof}[Proof of Proposition~\ref{pro:epsilonreg}]
 Let $\varepsilon > 0$ and
\begin{equation}
	\norm{v}_{L^{\xi_t}_t L^{\xi_x}_x(Q^\iota)} + \norm{q}_{L^{\zeta_t}_t L^{\zeta_x}_x(Q^\iota)} + \norm{a}_{\mathbf{M}_m(Q^\iota)} \leq \varepsilon.
\end{equation}
Let $\bar{\alpha} = \alpha_1/2$.

\textbf{1. Interior case}.
When $0 < \varepsilon \ll 1$, we may apply Lemma~\ref{lem:iteratedestimates} with $\alpha = \bar{\alpha}$ on $Q_{1/2}(z_0)$ for all $z_0 \in \overline{Q_{1/2}}$. The proof is completed by Campanato's characterization of H{\"o}lder continuity.

\textbf{2. Boundary case}. Ultimately, we wish to measure the oscillation of $v$ in $Q_r(z_0) \cap Q^+$ for all $z_0 \in \overline{Q_{1/2}^+}$ and $r \in (0,1/4]$. We denote by $d(z_0)$ the $x_3$-component of the space-time point $z_0$. For convenience, we write
\begin{equation}
	E = \norm{v}_{L^{\xi_t}_t L^{\xi_x}_x(Q^+)} + \norm{q}_{L^{\zeta_t}_t L^{\zeta_x}_x(Q^+)}.
\end{equation}

\emph{2a. Interior balls away from boundary strip} ($d(z_0) \geq 1/4$).
If $0 < \varepsilon \ll 1$, we may apply Lemma~\ref{lem:iteratedestimates} with $\alpha = \bar{\alpha}$ on $Q_{1/4}(z_0)$. Hence, for all $r \in (0,1/4]$, we have
\begin{equation}
	Y_\osc(z_0,r) \les r^{\bar{\alpha}} E \, .
\end{equation}

\emph{2b. Boundary balls} ($d(z_0) = 0$). If $0 < \varepsilon \ll 1$, we may apply Lemma~\ref{lem:iteratedestimates} with $\alpha = \bar{\alpha}$ on $Q^+_{1/2}(z_0)$. This yields, for all $r \in (0,1/2]$,
\begin{equation}
	\label{eq:boundaryestyplus}
	Y^+(z_0,r) + Y^+_\osc(z_0,r) \les r^{\bar{\alpha}} E \, .
\end{equation}

\emph{2c. Interior balls in the boundary strip} ($d(z_0) \in (0,1/4)$ and $r \leq d(z_0)$). Consider the case when the balls just touch the boundary. We control these balls via \emph{2b}:
\begin{equation}
    \label{eq:greatnameforareference}
	Y(z_0,d(z_0)) \les Y^+(z_0-d(z_0)e_3,2d(z_0)) \les d(z_0)^{\bar{\alpha}} E.
\end{equation}
That is, we double the radius of the ball and shift its center to the boundary.
If $0 < \varepsilon \ll 1$ (independently of $z_0$), we may apply Lemma~\ref{lem:iteratedestimates}, whose  hypotheses are guaranteed by~\eqref{eq:greatnameforareference}, on $Q_{d(z_0)}(z_0)$ with $\alpha = \bar{\alpha}$  to obtain that, for all $r \in (0,d(z_0)]$,
\begin{equation}\label{e.Yz0r}
	 Y_{\rm osc}(z_0,r) \les (r/d(z_0))^{\bar{\alpha}} Y(z_0,d(z_0)) \overset{\eqref{eq:greatnameforareference}}{\les} r^{\bar{\alpha}} E.
\end{equation}

\emph{2d. Balls intersecting the boundary} ($d(z_0) \in (0,1/4)$ and $r > d(z_0)$). Again, we can reduce to \emph{2b}. Specifically,
\begin{equation}
	r^{\#\xi} \norm{v - (v)_{Q_r(z_0) \cap Q^+}}_{L^{\xi_t}_t L^{\xi_x}_x(Q_r(z_0) \cap Q^+)} \les Y^+(z_0 - d(z_0) e_3,2r) \les r^{\bar{\alpha}} E \, .
\end{equation}

In summary, when $0 < \varepsilon \ll 1$, we have
\begin{equation}
	r^{\#\xi} \norm{v - (v)_{Q_r(z_0) \cap Q^+}}_{L^{\xi_t}_t L^{\xi_x}_x(Q_r(z_0) \cap Q^+)} \les r^{\bar{\alpha}} E
\end{equation}
for all $r \in (0,1/4]$ and $z_0 \in \overline{Q^+_{1/2}}$. Note that $c_1 r^5 \leq |Q_r(z_0) \cap Q^+| \leq c_2 r^5$. Since also $\norm{v}_{L^1(Q^+)} \les E$, Campanato's criterion yields
\begin{equation}
	\norm{v}_{C^{\bar{\alpha}}_\para(Q^+_{1/2})} \les E,
\end{equation}
as desired.
\end{proof}

\begin{proof}[Proof of Proposition~\ref{pro:epsilonregcritical}]
This is similar to the proof of Proposition~\ref{pro:epsilonreg} but simpler in that one no longer needs to estimate the oscillation or apply Campanato's criterion. We omit the details.
\end{proof}

\section{Proof of localized smoothing}

\begin{proof}[Proof of Theorem~\ref{thm:localsmoothinghalfspace} and Remark~\ref{rmk:subcriticalrefinements}]\

\emph{Step 0. Reduction to interior and boundary cases}. First, we argue that it suffices to demonstrate the theorem when $d(x_0) \geq 3$ (interior case) or $d(x_0) = 0$ (boundary case). The proof is by a covering argument:

(i) If $d(\Omega_1(x_0)) > 1$, one covers $\overline{\Omega_1(x_0)}$ with $O(1)$ balls $B_{1/6}(y_0)$ with centers $y_0 \in \overline{\Omega_1(x_0)}$. Notice that $B_{1/2}(y_0)$ remains inside $\Omega_3(x_0)$. Then we apply the rescaled interior case of the theorem on each $B_{1/2}(y_0)$.

(ii) If $d(\Omega_1(x_0)) \leq 1$, one covers the strip $\overline{\Omega_1(x_0)} \cap \{ d(x) \leq 1/12 \}$ by $O(1)$ boundary balls $B^+_{1/6}(y_0)$ with centers $y_0 = (y_0',0)$ satisfying $|y_0'- x_0'| \leq 1$. With this choice, $B^+_{1/2}(y_0)$ remains inside $\Omega_3(x_0)$. Away from the strip, one covers $\overline{\Omega_1(x_0)} \cap \{ d(x) \geq 1/12 \}$ by $O(1)$ balls $B_{1/36}(y_0)$ with $y_0 \in \overline{\Omega_1(x_0)} \cap \{ d(x) \geq 1/12 \}$, so that $B_{1/12}(y_0)$ remains inside $\Omega_3(y_0)$. Then we apply the rescaled interior and boundary cases of the theorem on the above interior and boundary balls.

Below, we write only the boundary case $x_0 \in \p \R^3_+$. The interior case is nearly identical.

\emph{Step 1. Small data}.
 Since $u$ is a local energy solution on $\R^3_+ \times (0,T)$, it satisfies the local energy and pressure estimates in Lemma~\ref{lem:localenergyestslemma} on $\R^3_+ \times (0,\bar{S}_1)$. 
Let $a_0$ be a divergence-free localization\footnote{Write $a_0 = \phi u_0 + w_0$, where $\phi$ is an appropriate smooth cut-off function and $w_0$, which corrects the non-zero divergence, is obtained using Bogovskii's operator~\cite{bogovskii}, see also Galdi's book~\cite[Section III.3]{galdi}.} with $a_0 \equiv u_0$ on $B_{2}^+(x_0)$, supported on $\overline{B_3^+(x_0)}$, with vanishing normal trace on $\p_{\rm flat} B_3^+(x_0)$, and such that
\begin{equation}
    \label{eq:a0getssmallerinnorm}
    \norm{a_0}_{L^p(\R^3_+)} \les_p \norm{u_0}_{L^p(B^+_3)} \text{ for all } p \in (1,m].
\end{equation}
 Proposition~\ref{pro:Lmsoltheory} guarantees that $a_0$ generates a mild solution $a$ of the Navier-Stokes equations satisfying the estimates~\eqref{eq:linearestimatem}-\eqref{eq:linearestimatelebesguem} on $\R^3_+ \times (0,T_m)$ with $u = a$. In particular, by Lemma~\ref{lem:acontrolledinMm}, there exists $0 < N_0 \ll 1$ satisfying that, whenever $N \leq N_0$, we have
\begin{equation}
	\norm{a}_{\mathbf{M}_m(B_{2}^+(x_0) \times (0,4))} \leq \varepsilon_{\rm CKN}/2.
\end{equation}
Let $N \leq N_0$. Let $z$ denote the pressure associated to $a$. Then $(a,z)$ is also a suitable Leray-Hopf solution with $\norm{a_0}_{L^2(\R^3_+)} \les M$; in particular, it is a local energy solution, so it satisfies the energy and pressure estimates in Lemma~\ref{lem:localenergyestslemma} on $\R^3_+ \times (0,\bar{S}_1)$ after possibly decreasing $\bar{S}_1$. Let
\begin{equation}
    \label{eq:vqdef}
	v = u - a, \quad q = p - z.
\end{equation} 
Then $(v,q)$ is a suitable weak solution of the perturbed Navier-Stokes equations~\eqref{eq:perturbednse} on $B_{2}^+(x_0) \times (0,\bar{S}_1)$ with coefficient $a$,\footnote{A technical point is to `transfer' the local energy inequality satisfied by the weak solution $u$ and the local energy \emph{equality} satisfied by the strong solution $a$ to the perturbed local energy inequality for $v$. This type of argument is well known from the proof of weak-strong uniqueness.} and since $u(\cdot,t),a(\cdot,t) \to u_0$ in $L^2(B^+_2(x_0))$ as $t \to 0^+$, $v$ satisfies
\begin{equation}\label{e.cvtzero}
	\norm{v(\cdot,t)}_{L^2(B_{2}^+(x_0))} \to 0 \text{ as } t \to 0^+ \, .
\end{equation}
Next, we use the estimates in Lemma~\ref{lem:localenergyestslemma} for the local energy solutions $u$ and $a$, the definition~\eqref{eq:vqdef} of $(v,q)$, and the triangle inequality to demonstrate that, for all $S \in (0,\bar{S}_1]$, we have
\begin{equation}
	\norm{v}_{L^{\xi_t}_t L^{\xi_x}_x(B^+_{2}(x_0) \times (0,S))} + \norm{q-[q]_{B_2^+(x_0)}}_{L^{\zeta_t}_t L^{\zeta_x}_x(B^+_{2}(x_0) \times (0,S))} \les S^{\frac{1}{\xi_t} - \frac{1}{p}} M^2 + S^{\frac{1}{\zeta_t} - \frac{3}{4}} M^2,
\end{equation}
 whenever $\#(\xi_x,p) = -3/2$ and $M \geq 1$. In particular, the exponents on the right-hand side may be taken positive. Therefore, we choose $S \in (0,S_1]$ satisfying
 \begin{equation}
    \label{eq:defofbarsforsmallN}
	S = O(1) M^{-O(1)} N^{O(1)} \,
 \end{equation}
 where $M \geq M_{\rm univ} > 0$ guarantees that $S \leq 1$, and, with $\bar{S} = \min(S,T)$,
 \begin{equation}
    \label{eq:thisguaranteetheckn}
	\norm{v}_{L^{\xi_t}_t L^{\xi_x}_x(B^+_{2}(x_0) \times (0,\bar{S}))} + \norm{q-[q]_{B_2^+(x_0)}}_{L^{\zeta_t}_t L^{\zeta_x}_x(B^+_{2}(x_0) \times (0,\bar{S}))} \leq \varepsilon_{\rm CKN} N/2.
 \end{equation}
 Thanks to~\eqref{e.cvtzero}, we may extend $(v,q)$ backward-in-time by zero as a suitable weak solution to~\eqref{eq:perturbednse} on $B_{2}^+ \times  (\bar{S}_1-4,\bar{S}_1)$. 
 Hence,~\eqref{eq:thisguaranteetheckn} guarantees that $(v,q)$ satisfies the $\varepsilon$-regularity criterion in Proposition~\ref{pro:epsilonreg} ($m>3$) or Proposition~\ref{pro:epsilonregcritical} ($m=3$) on $B_2^+(x_0) \times (\bar{S}-4,\bar{S})$. When $m > 3$, we have
 \begin{equation}
	\norm{v}_{C^{\bar{\alpha}}_\para(B^+(x_0) \times (\bar{S}-1,\bar{S}))} \les N.
 \end{equation}

 \emph{Step 2. Concluding for $m=3$.}
 When $m = 3$, we have a subcritical Morrey estimate rather than a $C^\alpha$ estimate. To demonstrate the critical time-weighted $L^\infty$ smoothing in Theorem~\ref{thm:localsmoothinghalfspace}, we will defer to the $\varepsilon$-regularity theory for the solution $(u,p)$ of the non-perturbed Navier-Stokes equations.  At this point, it will be convenient to specialize to $x_0 = 0$ without loss of generality.

The subcritical Morrey estimates (with $\alpha = -1/2$ in the statement of Proposition~\ref{pro:epsilonregcritical}) are
\begin{equation}
	\label{eq:Morreyestint}
	\sup_{z',R} RY(z',R,v,q) \les R^{1/2} N
\end{equation}
where $z' = (x',t')$, $|x'| < 1$, and $R < 1$ satisfy $Q_R(z') \subset Q_2^+((0,\bar{S}))$, and on the boundary,
\begin{equation}
	\label{eq:Morreyestbd}
	\sup_{z',R} RY^+(z',R,v,q) \les R^{1/2} N,
\end{equation}
where $|x'| < 1$, $d(x') = 0$, $R < 1$, and $Q^+_R(z') \subset Q_2^+((0,\bar{S}))$.

 We also require estimates on $(a,z)$. Since $L^3$ embeds into $L^2_{\rm uloc}$, these can be obtained from Lemma~\ref{lem:localenergyestslemma} and the scaling symmetry:
 \begin{multline}
 	\label{eq:azest}
	\sup_{\gamma \in (0,+\infty)} \sup_{x_1 \in \overline{\R^3_+}} \gamma^{\# \xi+1} \norm{a}_{L^{\xi_t}_t L^{\xi_x}_x(\Omega_{ \gamma }(x_1) \times (0,\gamma^2))}\\
	+ \gamma^{\# \zeta+2} \norm{z-[z]_{\Omega_{ \gamma}(x_1)}}_{L^{\zeta_t}_t L^{\zeta_x}_x(\Omega_{\gamma}(x_1) \times (0,\gamma^2))} \les N \, ,
 \end{multline}
 where $N \leq N_0$ is small enough to ensure that the strong $L^3$ solution exists globally and remains small.

 Let $t' \in (0,\bar{S})$ and $x'' \in B^+$. Recall that $\bar{S} \leq 1$. We consider a covering like the one in the proof of Proposition~\ref{pro:epsilonreg}: If $d(x'') > \sqrt{t'}/8$, then we consider the ball $Q_{\sqrt{t'}/8}(x',t')$ with $x' = x''$. If $d(x'') \leq \sqrt{t'}/8$, then we consider the half-ball $Q_{\sqrt{t'}/4}^+(x',t')$ with $x'=x''-d(x'')e_3$ (that is, the half-ball is spatially centered at the projection of $x''$ onto the flat boundary). In the former case, we appeal to~\eqref{eq:Morreyestint} and~\eqref{eq:azest}, whereas in the latter case, we appeal to~\eqref{eq:Morreyestbd} and~\eqref{eq:azest}. We write only the latter case. We have
 \begin{equation}
 	\label{eq:upNbound}
	\frac{\sqrt{t'}}{4} Y^+(z',\sqrt{t'}/4,u,p) \les N.
 \end{equation}
By choosing $N_0 \ll 1$, the right-hand side of~\eqref{eq:upNbound} can be made to satisfy the $\varepsilon$-regularity criterion in Proposition~\ref{pro:epsilonreg} for the non-perturbed Navier-Stokes equations~\eqref{e.nse}. Hence, $u$ is H{\"o}lder continuous in $Q^+_{\sqrt{t'}/8}(x',t')$,
\begin{equation}
	t'^{1/2} \| u \|_{L^\infty(Q^+_{\sqrt{t'}/8}(x',t'))} \les N
\end{equation}
and, in particular, 
\begin{equation}
	t'^{1/2} |u(x'',t')| \les N.
\end{equation}
Since $(x'',t') \in B^+ \times (0,\bar{S})$ was arbitrary, the proof for $m=3$ is complete.

 \emph{Step 3. Large data ($m > 3$)}. Without loss of generality, $x_0 = 0$ and $N \geq 4N_0$. We perform the following rescaling procedure. For $\lambda \in (0,1]$, we define the rescaled solutions
 \begin{equation}
	u_\lambda(x,t) = \lambda u(\lambda x,\lambda^2 t), \quad p_\lambda(x,t) = \lambda^2 p(\lambda x, \lambda^2 t).
 \end{equation}
 Let $u_{0,\lambda} = u_\lambda(\cdot,0)$. Then
 \begin{equation}
    \label{eq:illbeequaltoN0}
	\norm{u_{0,\lambda}}_{L^m(B^+_{3/\lambda})} \leq \lambda^{1-\frac{3}{m}} N.
 \end{equation}
 When
 \begin{equation}
    \label{eq:iaffordsomething}
	\lambda = \left( \frac{N_0}{N} \right)^{\frac{1}{1-\frac{3}{m}}},
 \end{equation}
the quantity on the right-hand side of~\eqref{eq:illbeequaltoN0} is equal to $N_0$. The rescaling also disrupts the local energy norm. That is, we have
\begin{equation}
	\norm{u_{0,\lambda}}_{L^2_{\uloc}(\R^3_+)} \leq \lambda^{-\frac{3}{2}} M =: M'.
\end{equation}
Next, we cover $\overline{B^+_{2/\lambda}}$ with balls $B_1(x_1)$ and $B_1^+(x_1)$ such that $x_1 \in \overline{B^+_{2/\lambda}}$. This requires $O(\lambda^{-3})$ balls, each with $O(1)$ intersections. In the balls $B_3(x_1)$ and $B_3^+(x_1)$, which are contained in $B^+_{3/\lambda}$ since $1/\lambda \geq 4$, we apply the small-data smoothing demonstrated above, which is afforded by~\eqref{eq:illbeequaltoN0} and~\eqref{eq:iaffordsomething}. To complete the $p = m$ smoothing estimate, we sum
\begin{equation}
    \sup_{t \in I} \int_{B_1^\iota(x_1)} |u_\lambda(x,t)|^m \, dx \les \int_{B_3^\iota(x_1)} |u_{0,\lambda}|^m \, dx
\end{equation}
over the covering of balls, where $\iota \in \{ {\rm int}, {\rm bd} \}$ depending on $x_1$, and $I$ is the time interval from the application of the small-data smoothing on each ball. For the $p = +\infty$ smoothing estimate, there is no need to sum. For $p \in (m,+\infty)$, we interpolate between the $p=m$ and $p=+\infty$ smoothing estimates. Finally, undoing the rescaling yields the theorem, where now $S = O(1) M^{-O(1)} N^{-O(1)}$ for $M \geq M_{\rm univ}$ and $N \geq 4N_0$ (compare to~\eqref{eq:defofbarsforsmallN}, which is valid for $N \leq N_0$). \end{proof}

\begin{proof}[Proof of Theorem~\ref{thm:localsmoothinglocal}]
Without loss of generality, we may replace $\zeta_t$ in the statement of Theorem~\ref{thm:localsmoothinglocal} by $\zeta_t / (1- \nu \zeta_t)$, where $0 < \nu \ll 1$. That is, we assume
\begin{equation}
	\label{eq:muhassumptiononuandp}
	 \| u \|_{L^\infty_t L^2_x(\Omega_3(x_0) \times (0,T))} + \| \nabla u \|_{L^2_{t,x}(\Omega_3(x_0) \times (0,T))} + \| p \|_{L^{\zeta_t / (1- \nu \zeta_t)}_t L^{\zeta_x}_x(\Omega_3(x_0) \times (0,T))} \leq M.
\end{equation}
Then the proof is identical to the proof of Theorem~\ref{thm:localsmoothinghalfspace} up to a minor adjustment, namely, that~\eqref{eq:muhassumptiononuandp} is used to control $(u,p)$ rather than the local energy estimates in Lemma~\ref{lem:localenergyestslemma}. The strong solution $a$ is controlled in the same way as before, so
 by the triangle inequality, we can still obtain
\begin{equation}
	\norm{v}_{L^{\xi_t}_t L^{\xi_x}_x(B^+_{2}(x_0) \times (0,S))} + \norm{q-[q]_{B_2^+(x_0)}}_{L^{\zeta_t}_t L^{\zeta_x}_x(B^+_{2}(x_0) \times (0,S))} \les_{\nu} S^{\nu} M^2 \, .
\end{equation}
\end{proof}

\section{Proof of concentration}
The goal of this subsection is to prove Theorem~\ref{thm:localizedconcentration}. First, it is necessary to introduce some notation.
Let $\mathcal{O}$ be a domain in $\R^3$. We define the following scale-invariant quantities, which will be used throughout our work: for $x\in\overline{\mathcal{O}}$ and $r\in(0,\infty)$,
\begin{align}
\label{Adefx}
A(u,r;\mathcal{O},x,t):=\ &\sup_{t-r^2<s<t}\frac{1}{r}\int\limits_{B_r(x)\cap\mathcal{O}} | u(y,s)|^{2} dy,\\
\label{Edefx}
E(u,r;\mathcal{O},x,t):=\ &\frac{1}{r} \int\limits_{t-r^2}^t\int\limits_{B_r(x)\cap\mathcal{O}} |\nabla u|^{2} dy\, ds,\\
\label{scaledpressurelambda32x}
D_{\frac32}(p,r;\mathcal{O},x,t):=\ &\frac{1}{r^{2}}\int\limits_{t-r^2}^t\int\limits_{B_r(x)\cap\mathcal{O}} |p-[p(\cdot,s)]_{B_r(x)\cap\mathcal{O}}|^{\frac32} dy \, ds,\\
\label{scaledpressurezeta}
D_{\zeta_{x},\zeta_{t}}(p,r;\mathcal{O},x,t):=\ &\frac{1}{r^{(\frac{3}{2}-\delta_{0})\zeta_{t}}}\int\limits_{t-r^2}^t\Big(\int\limits_{B_r(x)\cap\mathcal{O}} |p-[p(\cdot,s)]_{B_r(x)\cap\mathcal{O}}|^{\zeta_{x}} dy\Big)^{\frac{\zeta_{t}}{\zeta_{x}}}ds
\end{align}

Here, as in the rest of this paper, 
$(\zeta_{x},\zeta_{t})$ is as in Section 2 and
\begin{equation}
    \notag
    [f]_{\mathcal{O}}:= \frac{1}{ |\mathcal{O}|} \int\limits_{\mathcal{O}} f(y) dy.
\end{equation}
Below, we will often take $(x,t)=(0,0)$. 
In this case, we have the following lighter notation:
\begin{equation}
\label{Adef}
    \begin{aligned}
&A(u,r; \mathcal{O}):=A(u,r;\mathcal{O},0,0), &&E(u,r; \mathcal{O}):=E(u,r;\mathcal{O},0,0),\\
&D_{\frac32}(p,r;\mathcal{O}):=D_{\frac32}(p,r;\mathcal{O},0,0),
 &&D_{\zeta_{x},\zeta_{t}}(p,r;\mathcal{O}):=D_{\zeta_{x},\zeta_{t}}(p,r;\mathcal{O},0,0).
    \end{aligned}
\end{equation}

The main ingredients in proving Theorem~\ref{thm:localizedconcentration} is Theorem \ref{thm:localsmoothinglocal}, combined with a rescaling argument and the following key proposition (which we now state).
\begin{proposition}\label{scaleinvarestnotcentre}
Let $(u,p)$ be a suitable weak solution of~\eqref{e.nse} on $Q_4^+$, in the sense of Definition~\ref{def.sws}, satisfying 
\begin{equation}\label{boundedenergyonescalepro}
 \| \nabla u \|_{L^2_{t,x}(Q_4^+)} + \| p \|_{L^{\frac{3}{2}}_{t,x}(Q_4^+)} \leq M_{0}
\end{equation}
and
\begin{equation}\label{boundedkineticenergyallscalespro}
\sup_{(y_0,s_{0})\in Q_3^+}\sup_{0<r\leq 1}A(u,r;B_4^+,y_0,s_0)\leq A_{0}.
\end{equation}
Then the above assumptions imply that
 \begin{align}\label{scaledenergypressureallscalesuncentre}
 \begin{split}
 \sup_{(y_0,s_{0})\in Q^+}\sup_{0<r\leq 1} \left\lbrace E(u,r;B_4^+,y_0,s_0)+D_{\zeta_{x},\zeta_{t}}(p,r;B_4^+,y_0,s_0) \right\rbrace \leq M(M_{0},A_{0}).
 \end{split}
 \end{align}
\end{proposition}
The first part of this section focuses on proving Proposition~\ref{scaleinvarestnotcentre}.  In doing so, we will need the following proposition, which is a rescaled version of results taken from \cite{Ser09}.

\begin{proposition}[Maximal pressure regularity]\label{localboundaryregstokes}
Let $m,n$ and $s$ be such that $1<m<\infty$, $1<n\leq \infty$ and $m\leq s<\infty$. Suppose $\nabla u\in L^{n}_{t}L^{m}_{x}(Q^+_r)$, $p\in L^{n}_{t}L^{m}_{x}(Q^+_r)$ and $f\in L^{n}_{t}L^{s}_{x}(Q^+_r)$.
In addition, suppose that
\begin{equation}
\partial_{t}u-\Delta u+\nabla p=f,\qquad \div u=0\qquad \mbox{in}\quad Q^{+}_r \, ,    
\end{equation}
and suppose $u$ satisfies the boundary condition 
\begin{equation}
    u=0 \quad \text{ on } \quad x_{3}=0.
\end{equation}
Then, we conclude that $\nabla p\in L^{n}_{t}L^{s}_{x}(Q^{+}_{r/2})$. Furthermore, the estimate
\begin{equation}
    \begin{aligned}
    &\|\nabla p\|_{L^{n}_{t}L^{s}_{x}(Q^{+}_{r/2})}
\leq c(s,n,m) \big(\|f\|_{L^{n}_{t}L^{s}_{x}(Q^+_r)} 
+r^{\frac{3}{s}-\frac{3}{m}-2}\| u\|_{L^{n}_{t}L^{m}_{x}(Q^+_r)}\\
 &\quad + r^{\frac{3}{s}-\frac{3}{m}-1}(\|\nabla  u\|_{L^{n}_{t}L^{m}_{x}(Q^+_r)}+\|p-[p]_{B^{+}_r}\|_{L^{n}_{t}L^{m}_{x}(Q^+_r)}) \big)
    \end{aligned}
\end{equation}
holds.
\end{proposition}
Another key step in proving Proposition \ref{scaleinvarestnotcentre} is establishing the following simplified versions for balls centered at the space-time point $(0,0)$, which we now state as two separate propositions.
 
\begin{proposition}\label{centeredscaleinvarestboundary}
Let $(u,p)$ be a suitable weak solution of~\eqref{e.nse} on $Q^{+}$, in the sense of Definition~\ref{def.sws},  satisfying 
\begin{equation}\label{boundedenergyonescalecentered}
 (u,p)\in L^{\infty}_{t}L^{2}_{x}\cap L^{2}_{t}\dot{H}_x^{1}(Q^{+})\times L^{\frac{3}{2}}(Q^{+}).
\end{equation}
 Suppose that $u$ satisfies 
 \begin{equation}\label{centeredscaleinvariantKE}
 \sup_{0<r\leq 1}A(u,r;B^+)\leq A_{0}.
 \end{equation}
 Then the above assumptions imply that
 \begin{align}\label{scaledenergypressureallscalescentre}
 \begin{split}
 \sup_{0<r\leq 1}& \left\lbrace E(u,r;B^+)+D_{\frac{3}{2}}(p,r; B^+)+r^{-\frac{3}{2}+\delta_{0}}\|\nabla p\|_{L^{\zeta_{t}}_{t}L^{\frac{3\zeta_{x}}{3+\zeta_{x}}}_{x}(Q^+_r)} \right\rbrace \\
 &\leq F(A_{0}, D_{\frac{3}{2}}(p,1;B^+),E(u,1;B^+) ),
 \end{split}
 \end{align}
 for a function $F$ increasing in its (three) arguments.
\end{proposition}

\begin{proof}[Proof of Proposition \ref{centeredscaleinvarestboundary}]First, by Lemma 3.2 of~\cite{Mik09},
we see that the assumptions \eqref{boundedenergyonescalecentered}-\eqref{centeredscaleinvariantKE} imply that
\begin{equation}
\label{mikhaylov}
    \begin{aligned}
 &\sup_{0<r\leq 1} \left\lbrace E(u,r;B^+)+D_{\frac{3}{2}}(p,r; B^+) \right\rbrace
 \leq F^{(1)}(A_{0}, D_{\frac{3}{2}}(p,1;B^+),E(u,1;B^+) ).
    \end{aligned}
\end{equation}
 Using that $u(\cdot,t)$ vanishes on $\p B^+\cap\p\mathbb{R}^3_{+}$, we can apply H\"{o}lder's inequality and Poincar\'{e}'s inequality to infer that
 \begin{equation}\label{nonlinearityest}
 \|u\cdot\nabla u\|_{L^{\frac{4}{3}}_{t}L^{\frac{6}{5}}_{x}(Q^+_r)}\leq C\|u\|_{L^{\infty}_{t}L^{2}_{x}(Q^+_r)}^{\frac{1}{2}}\|\nabla u\|_{L^{2}(Q^+_r)}^{\frac{3}{2}}.
 \end{equation}
 We use this to  now apply Proposition \ref{localboundaryregstokes} with $n=\frac{4}{3}$, $s=m=\frac{6}{5}$. This and H\"{o}lder's inequality allow us to infer that for $0<r\leq 1$,
 \begin{equation}
 \label{locmaxregapplication}
     \begin{aligned}
       &r^{-1}\|\nabla p\|_{L^{\frac{4}{3}}_{t}L^{\frac{6}{5}}_{x}(Q^{+}_{r/2})}\leq  C (r^{-\frac{1}{2}}\|u\|_{L^{\infty}_{t}L^{2}_{x}(Q^+_r)})^{\frac{1}{2}}(r^{-\frac{1}{2}}\|\nabla u\|_{L^{2}(Q^+_r)})^{\frac{3}{2}}\\
 &\qquad + Cr^{-\frac{1}{2}} ( \|\nabla u\|_{L^{2}(Q^+_r)}+\|u\|_{L^{\infty}_{t}L^{2}_{x}(Q^+_r)} )\\
 &\qquad + Cr^{-\frac{4}{3}}\|p-(p)_{B^{+}(r)}\|_{L^{\frac{3}{2}}_{x,t}(Q^+_r)}.
     \end{aligned}
 \end{equation}
 Since $\zeta_{x}<2$ and $\zeta_{t}< 4/3$, H\"{o}lder's inequality gives
 \begin{equation}\label{holdergradp}
 r^{-\frac{3}{2}+\delta_{0}}\|\nabla p\|_{L^{\zeta_{t}}L^{\frac{3\zeta_{x}}{3+\zeta_{x}}}(Q^{+}_{r/2})}\leq Cr^{-1}\|\nabla p\|_{L^{\frac{4}{3}}_{t}L^{\frac{6}{5}}_{x}(Q^{+}_{r/2})}.
 \end{equation}
 Here, $C$ can be taken to be independent of $\zeta_{x}$ and $\zeta_{t}$.
 
 Now combining \eqref{centeredscaleinvariantKE} with \eqref{mikhaylov}, \eqref{locmaxregapplication} and \eqref{holdergradp} readily gives the desired conclusion \eqref{scaledenergypressureallscalescentre}.
\end{proof}
To prove Proposition \ref{scaleinvarestnotcentre}, we also require the following interior analogue of Proposition \ref{centeredscaleinvarestboundary}.

\begin{proposition}\label{centeredscaleinvarestinterior}
Let $(u,p)$ be a suitable weak solution of~\eqref{e.nse} on $Q$, in the sense of Definition~\ref{def.sws},  satisfying 
\begin{equation}\label{boundedenergyonescalecenteredinterior}
 (u,p)\in L^{\infty}_{t}L^{2}_{x}\cap L^{2}_{t}\dot{H}^1_x(Q)\times L^{\frac{3}{2}}(Q).
\end{equation}
Suppose that $u$ satisfies 
 \begin{equation}\label{centeredscaleinvariantKEinterior}
 \sup_{0<r\leq 1}A(u,r;B)\leq A_{0}.
 \end{equation}
 Then the above assumptions imply that
 \begin{equation}
 \label{scaledenergypressureallscalescentreinterior}
      \begin{aligned}
 \sup_{0<r\leq 1}& \left\lbrace E(u,r;B)+D_{\frac{3}{2}}(p,r; B)+r^{-\frac{3}{2}+\delta_{0}}\|\nabla p\|_{L^{\zeta_{t}}_{t}L^{\frac{3\zeta_{x}}{3+\zeta_{x}}}_{x}(Q_r)} \right\rbrace\\
 &\leq F(A_{0}, D_{\frac{3}{2}}(p,1;B),E(u,1;B) ).
 \end{aligned}
 \end{equation}
\end{proposition}

The proof of Proposition~\ref{centeredscaleinvarestinterior} is nearly identical to that of Proposition~\ref{centeredscaleinvarestboundary} and hence is omitted. We remark that to prove Proposition~\ref{centeredscaleinvarestinterior}, one uses the interior analogue of Mikhaylov's result~\cite{Mik09} proven by Seregin~\cite{Ser06}.

\begin{proof}[Proof of Proposition \ref{scaleinvarestnotcentre}] First, we note that by Poincar\'{e}'s inequality it suffices to show that
\begin{equation}
\label{scaledenergygradpressureallscalesuncentre}
    \begin{aligned}
 &\sup_{(y_0,s_{0})\in Q^+}\sup_{0<r\leq 1} \Big\lbrace E(u,r;B^+_4,y_0,s_0) \\
 &\quad + r^{-\frac{3}{2}+\delta_{0}}\|\nabla p\|_{L^{\zeta_{t}}_{t}L^{\frac{3\zeta_{x}}{3+\zeta_{x}}}_{x}(\Omega_r(x_0) \times (s_0-r^2,s_0))} \Big\rbrace \leq \bar{M}(M_{0},A_{0}).
    \end{aligned}
\end{equation}
 Now, we fix $(x_0,t_0)\in \overline{B^{+}} \times (-1,0)$ and $0<r\leq 1$. We divide into the various cases that arise.
 
 \textbf{ 1. Boundary case: $d(x_0)=0$.} In this case, we can directly apply a translated version of Proposition \ref{centeredscaleinvarestboundary}. This gives
 \begin{align}
 \begin{split}
 & E(u,r;B_4^+,x_0,t_0)+D_{\frac{3}{2}}(p,r; B_4^+, x_0, t_0)
 \\&\quad +r^{-\frac{3}{2}+\delta_{0}}\|\nabla p\|_{L^{\zeta_{t}}_{t}L^{\frac{3\zeta_{x}}{3+\zeta_{x}}}_{x}(\Omega_r(x_0) \times (t_0-r^2,t_0))}\\
 &\quad\quad \leq
  F(A_{0},D_{\frac{3}{2}}(p,1; B_4^+, x_0,t_0), E(u,1; B_4^+, x_0,t_0)). 
 \end{split}
 \end{align}
 This estimate is of the form \eqref{scaledenergygradpressureallscalesuncentre}.
 
 \textbf{2. Balls intersecting with the boundary: $r\geq d(x_0).$} 
 In this case, we have $\Omega_r(x_0) \subset B^{+}_{r+d(x_0)}(x_0- d(x_0)e_3)$ and we can then argue as in `1. Boundary case.'
 
 \textbf{3. Balls not intersecting with the boundary: $r< d(x_0)$.} 
 This is the most involved of the three cases and involves two steps. First, we apply a translated and rescaled version of Proposition~\ref{centeredscaleinvarestinterior} to obtain
 \begin{align}
 \begin{split}
 & E(u,r;B_4^+,x_0,t_0)+D_{\frac{3}{2}}(p,r; B_4^+, x_0,t_0)\\
 &\quad +r^{-\frac{3}{2}+\delta_{0}}\|\nabla p\|_{L^{\zeta_{t}}_{t}L^{\frac{3\zeta_{x}}{3+\zeta_{x}}}_{x}(\Omega_r(x_0) \times (t_0-r^2,t_0))}\\
 &\quad\quad \leq 
  F(A_{0},D_{\frac{3}{2}}(p,d(x_0); B_4^+, x_0,t_0), E(u,d(x_0); B_4^+, x_0,t_0)). 
 \end{split}
 \end{align}
 We then control $D_{\frac{3}{2}}(p,d(x_0); B_4^+, x_0,t_0)$ and $E(u,d(x_0); B_4^+, x_0,t_0))$ by appealing to `2. Balls intersecting with the boundary.' This gives an estimate of the form~\eqref{scaledenergygradpressureallscalesuncentre}. \end{proof}
 
\begin{remark}\label{Mdependence}
Observing the statements of Seregin's result \cite{Ser06} and Mikhaylov's result \cite{Mik09}, it is not difficult to determine the dependence of $M(M_0,A_0)$ in Proposition~\ref{scaleinvarestnotcentre}. In particular, $M(M_0,A_0)$ can be taken to be polynomial in~$M_0$ and~$A_0$.
\end{remark}

\begin{proof}[Proof of Theorem~\ref{thm:localizedconcentration}]
First, recall $M=M(M_0,A_0)$ is as in Proposition \ref{scaleinvarestnotcentre}. Additionally, $S(M) = S(M,N_0,3) \in (0,1]$ and $N_0$ are as in Theorem~\ref{thm:localsmoothinglocal}. 
Furthermore, we define 
\begin{equation}\label{t*def}
\bar{t}(M):= -\frac{1}{9}S(M)
\end{equation}
and from now on we consider $t\in [\bar{t}(M),0)$.
With such choices, it is clear that
\begin{equation}
    	R(t) := 3 \times \sqrt{\frac{-t}{S(M)}}\in (\sqrt{-t},1)\,\,\,\,\,\,\forall t\in [\bar{t}(M),0).
\end{equation}
With these parameters fixed, the proof of Theorem \ref{thm:localsmoothinglocal} is by contraposition.
We assume that for any fixed $t\in[\bar{t}(M),0)$,

\begin{equation}\label{smallnessL3}
\|u(\cdot,t)\|_{L^{3}(\Omega_{R(t)}(x^*))}\leq N_{0} 
\end{equation}
 and show that this implies that $(x^*,0)$ is \textit{not} a singular point of $u$. Here $x^*\in \overline{{B}^{+}}$. 
 
 First, note that \eqref{boundedenergyonescalethm}, \eqref{boundedkineticenergyallscalesthm} and Proposition \ref{scaleinvarestnotcentre} imply that
 \begin{align}\label{energypressurescaleR(t)}
 \begin{split}
 &R(t)^{-\frac{1}{2}}\|u\|_{L^{\infty}_{t}L^{2}_{x}\cap L^{2}_{t}\dot{H}^1_x(\Omega_{R(t)} (x^*)\times (-R(t)^2,0))}\\
 &\quad +R(t)^{-\frac{3}{2}+\delta_{0}}\|p-[p]_{\Omega_{R(t)}(x^*)}\|_{L^{\zeta_{t}}_{t}L^{\zeta_{x}}_{x}(\Omega_{R(t)} (x^*)\times (-R(t)^2,0))} \leq M(M_0, A_0).
 \end{split}
 \end{align}
 Now, we define
 \begin{equation}\label{lambdadef}
 \lambda:=\sqrt{\frac{-t}{S(M)}}= \frac{1}{3} R(t)<\frac{1}{3}
 \end{equation}
 and the rescalings
 \begin{equation}\label{rescalings}
 (u_{\lambda}(x,s), p_{\lambda}(x,s)):=(\lambda u(\lambda x, \lambda^2 s+t), \lambda^2 p(\lambda x, \lambda^2 s+t))\quad x^*_{\lambda}:= \lambda x^*.
 \end{equation}
 It is clear that the point $(x,s)=(x^*_{\lambda}, S(M))$ for $(u_{\lambda}, p_{\lambda})$ corresponds to $(x^*,0)$ for the unscaled $(u,p)$. From \eqref{smallnessL3}-\eqref{energypressurescaleR(t)} we have
 \begin{equation}\label{smallnessL3rescaled}
 \|u_{\lambda}(\cdot,0)\|_{L^{3}(\Omega_{3}(x^*_{\lambda}))}\leq N_{0}
 \end{equation}
 and 
 \begin{equation}\label{rescaledenergypressure}
 \begin{split}
 &\|u_{\lambda}\|_{L^{\infty}_{t}L^{2}_{x}\cap L^{2}_{t}\dot{H}^1_x(\Omega_{3} (x^*_{\lambda})\times (0, S(M)))}+\|p_{\lambda}-[p_{\lambda}]_{\Omega_{3}(x^*_{\lambda})}\|_{L^{\zeta_{t}}_{t}L^{\zeta_{x}}_{x}(\Omega_{3} (x^*_{\lambda})\times (0, S(M)))}\leq M.
 \end{split}
 \end{equation}
 We can then apply Theorem \ref{thm:localsmoothinglocal} to infer that $(x^*_{\lambda},S(M))$ is a regular point of $u_{\lambda}$. Undoing the rescaling, we see that this implies that $(x^*,0)$ is a regular point for~$u$. \end{proof}

\subsubsection*{Acknowledgments}
DA was supported by NSF Postdoctoral Fellowship  Grant No. 2002023 and Simons Foundation Grant No. 816048. DA is also grateful to ENS Paris for partially supporting his academic visit to Paris during which this research was initiated.  CP is partially supported by the Agence Nationale de la Recherche, project BORDS, grant ANR-16-CE40-0027-01, project SINGFLOWS, grant ANR-18-CE40-0027-01, project CRISIS, grant ANR-20-CE40-0020-01 and by the CY Initiative of Excellence, project CYNA.

\begin{appendix}
\numberwithin{theorem}{section}

\section{$L^m$ solution theory}

Here we collect statements about the well-known perturbation theory for the Navier-Stokes equations in the half-space, with contributions due to~\cite{McCracken,Weissler,kato,Giga1985,Giga1986} and many others.

\begin{proposition}[$L^m$ solution theory]
	\label{pro:Lmsoltheory}
	Let $u_0 \in L^m_\sigma(\R^3_+)$ with $m \in [3,+\infty)$ and $\norm{u_0}_{L^m_\sigma(\R^3_+)} \leq N$.
	\begin{itemize}[leftmargin=*]
		\item (Subcritical) If $m > 3$, then there exists $T_m = T_m(N) > 0$ and a mild solution $u \in C([0,T_m];L^m(\R^3_+))$ satisfying, for all $p \in [m,+\infty]$,
		\begin{equation}
			\label{eq:linearestimatem}
			\sup_{t \in (0,T_m)} t^{\frac{3}{2}\left(\frac{1}{m}-\frac{1}{p} \right)} \norm{u(\cdot,t)}_{L^p(\R^3_+)} \les_{m} N
		\end{equation}
		\begin{equation}
			\label{eq:linearderivestm}
			\sup_{t \in (0,T_m)} t^{\frac{3}{2}\left(\frac{1}{m}-\frac{1}{p} \right)+\frac{1}{2}} \norm{\nabla u(\cdot,t)}_{L^p(\R^3_+)} \les_{m} N
		\end{equation}
		\begin{equation}
			\label{eq:linearestimatelebesguem}
			\norm{u}_{L^{5m/3}_{t,x}(\R^3_+ \times (0,T_m))} \les_{m} N.
		\end{equation}
		The mild solution is unique in the class $C([0,T];L^m_\sigma(\R^3_+))$.
		\item (Critical) If $m=3$, then there exists $T_3 = T_3(u_0) > 0$ and a mild solution $u \in C([0,T_3];L^3(\R^3_+))$ satisfying~\eqref{eq:linearestimatem}, \eqref{eq:linearderivestm}, and~\eqref{eq:linearestimatelebesguem}.
		If $N \ll 1$, then $T_3 = +\infty$.
		The mild solution is unique in the class $u \in L^5_{t,x}(\R^3_+ \times (0,T_3))$.
		\item If $u_0 \in L^{m_1}_\sigma(\R^3_+) \cap L^{m_2}_\sigma(\R^3_+)$ with $m_1, m_2 \in [3,+\infty)$, then the mild solutions guaranteed by the above points are identical.
		\item (Weak-strong uniqueness) If also $u_0 \in L^2_\sigma(\R^3_+)$, then all weak Leray-Hopf solutions are identical to the above mild solution on its existence time.
	\end{itemize}
\end{proposition}

The mild formulation of the Navier-Stokes equation is, formally,
\begin{equation}
    \label{eq:mildformulation}
\begin{aligned}
    u(\cdot,t) &= e^{t A} u_0 - \int_0^t e^{(t-s) A} \bP \div (u \otimes u)(\cdot,s) \, ds \\
    &= e^{t A} u_0 - \int_0^t e^{(t-s) A} \bP (u \cdot \nabla u)(\cdot,s) \, ds.
    \end{aligned}
\end{equation}
We summarize below only the linear theory necessary to prove Proposition~\ref{pro:Lmsoltheory}. To do so, one may follow~\cite[Chapter 5]{Tsaibook}.

\begin{proof}[Summary of linear theory]
 For all $m \in (1,+\infty)$, the Stokes operator
\begin{equation}
    A = \bP \Delta : (W^{2,m} \cap W^{1,m}_0 \cap L^m_\sigma)(\R^3_+) \subset L^m_\sigma(\R^3_+) \to L^m_\sigma(\R^3_+)
\end{equation}
 generates an analytic semigroup $t \mapsto S(t)$ in $L^m(\R^3_+)$, $t \geq 0$. In particular, $S(t) u_0 \in D(A^k)$ for all $k \geq 0$ and $t > 0$, and the following estimates are satisfied:
\begin{equation}
    t^k \| A^k S(t) u_0 \|_{L^m_\sigma(\R^3_+)} \les_m \| u_0 \|_{L^m_\sigma(\R^3_+)}.
\end{equation}
By elliptic regularity for the steady Stokes equations, we have that
\begin{equation}
    t^k \| \nabla^{2k}_x S(t) u_0 \|_{L^m(\R^3_+)} \les_m \| u_0 \|_{L^m_\sigma(\R^3_+)}
\end{equation}
 for all $k \geq 0$. Upon interpolating, we have
\begin{equation}
    \label{eq:myappendixestimate}
    t^{\frac{1}{2}} \| \nabla_x S(t) u_0 \|_{L^p(\R^3_+)} + \| S(t) u_0 \|_{L^p_\sigma(\R^3_+)} \les_m t^{\frac{1}{2} (\frac{3}{p} - \frac{3}{m})} \| u_0 \|_{L^m_\sigma(\R^3_+)}
\end{equation}
for all $p \in [m,+\infty]$. In particular,~\eqref{eq:myappendixestimate} estimates $\nabla_x S(t) \bP \: L^m(\R^3_+;\R^3) \to L^m(\R^3_+;\R^{3\times3})$.  By duality, we have
\begin{equation}
    \label{eq:muhdualityest}
    t^{\frac{1}{2}} \| S(t) \bP \div F \|_{L^m(\R^3_+;\R^3)} \les_m \| F \|_{L^m(\R^3_+;\R^{3\times3})}
\end{equation}
for all $m \in (1,+\infty)$. Alternatively, one may argue by means of the operator $A^{1/2}$.
By combining~\eqref{eq:muhdualityest} with~\eqref{eq:myappendixestimate}, we have
\begin{equation}
    \label{eq:muhwonderfuldualityest}
    t^{\frac{1}{2}} \| S(t) \bP \div F \|_{L^p(\R^3_+;\R^3)} \les_m t^{\frac{1}{2} (\frac{3}{p} - \frac{3}{m})} \| F \|_{L^m(\R^3_+;\R^{3\times3})}.
\end{equation}
for all $p \in [m,+\infty]$. This completes the proof of the time-weighted estimates.

We now argue the space-time Lebesgue estimates
\begin{equation}
    \label{eq:spacetimeestapp1}
    \| S(t) u_0 \|_{L^{5m/3}_{t,x}(\R^3_+ \times \R_+)} \les_m \| u_0 \|_{L^m_\sigma(\R^3_+)} \, ,
\end{equation}
for $m \in (1,+\infty)$, and
\begin{equation}
\label{eq:spacetimeestapp2}
\begin{aligned}
    &\left\| \int_0^t S(t-s) \bP \div F(\cdot,s) \, ds \right\|_{L^{5m/3}_{t,x}(\R^3_+ \times (0,T);\R^3)} \\
    &\quad \les_m T^{\frac{1}{2}(1-\frac{3}{m})} \| F \|_{L^{5m/6}_{t,x}(\R^3_+ \times (0,T);\R^{3\times3})} \, ,
    \end{aligned}
\end{equation}
for $m \in [3,+\infty)$. We follow~\cite{Giga1986}.
To see~\eqref{eq:spacetimeestapp1}, consider the sublinear operator
\begin{equation}
    u_0 \mapsto \left( t \mapsto \| S(t) \bP u_0 \|_{L^{5m/3}_x} \right) : L^m(\R^3_+;\R^3) \to L^{5m/3,\infty}(0,+\infty),
\end{equation}
where the mapping property is due to~\eqref{eq:myappendixestimate}. By the Marcinkiewicz interpolation theorem, we may interpolate between different values of $m$ to obtain~\eqref{eq:spacetimeestapp1}.\footnote{To apply the Marcinkiewicz theorem, it is crucial that $5m/3 \geq m$.} Finally,~\eqref{eq:spacetimeestapp2} can be obtained using~\eqref{eq:muhwonderfuldualityest} and the Hardy-Littlewood-Sobolev inequality (alternatively, Young's convolution inequality in Lorentz spaces).
\end{proof}

\end{appendix}

\parskip	0mm

\bibliographystyle{alpha}
\bibliography{localsmoothingbib}

\end{document}